\documentclass[12pt]{article}
\usepackage{no-ejc}

\usepackage{amsmath,amssymb}
\usepackage{graphicx}
\usepackage{hyperref}

\dateline{}{}{}


\newcommand{\R}{\mathbb{R}}
\renewcommand{\SS}{\mathbb{S}}

\newcommand{\set}[1]{\left\{#1\right\}}

\usepackage{tikz}
\usetikzlibrary{arrows}
\usetikzlibrary{shapes,snakes}
\usepackage{geometry}
\def\tn{\textnormal}
\def\ld{\lambda}
\def\mR{\mathbb{R}}
\def\mS{\mathbb{S}}

\def\checkmark{\tikz\fill[scale=0.4](0,.35) -- (.25,0) -- (1,.7) -- (.25,.15) -- cycle;} 

\makeatletter
\def\hlinewd#1{%
  \noalign{\ifnum0=`}\fi\hrule \@height #1 \futurelet
   \reserved@a\@xhline}
\makeatother

\def\A{\mathcal{A}}
\def\B{\mathcal{B}}

\def\Hm{\mathcal{H}}
\def\cG{\mathcal{G}}
\def\I{\mathcal{I}}
\def\J{\mathcal{J}}
\def\K{\mathcal{K}}
\def\M{\mathcal{M}}

\def\S{\mathcal{S}}

\def\tn{\textnormal}

\def\es{\emptyset}

\def\sm{\setminus}

\def\a{\alpha}
\def\b{\beta}
\def\l{\ell}

\newcommand{\ignore}[1]{}

\DeclareMathOperator{\LS}{LS}

\DeclareMathOperator{\Las}{Las}

\DeclareMathOperator{\FRAC}{FRAC}
\DeclareMathOperator{\STAB}{STAB}
\DeclareMathOperator{\MT}{MT}
\DeclareMathOperator{\bMT}{bMT}
\DeclareMathOperator{\THG}{TH}
\DeclareMathOperator{\Aut}{Aut}
\DeclareMathOperator{\sat}{sat}

\makeatletter
\def\hlinewd#1{%
  \noalign{\ifnum0=`}\fi\hrule \@height #1 \futurelet
   \reserved@a\@xhline}
\makeatother

\Copyright{  The authors. Released under the CC BY-ND license (International 4.0).}

\title{On Connections Between Association Schemes and \\ 
Analyses of Polyhedral and Positive Semidefinite \\
Lift-and-Project Relaxations}

\ignore{
\title{Matchings, hypergraphs, association schemes, \\ and semidefinite optimization}
}

\author{Yu Hin (Gary) Au \\
\small Department of Mathematics and Statistics\\[-0.8ex]
\small University of Saskatchewan\\[-0.8ex] 
\small Saskatoon, Canada\\
\small\tt au@math.usask.ca \\
\and
Nathan Lindzey \\
\small Department of Mathematical Sciences\\[-0.8ex]
\small University of Memphis\\[-0.8ex]
\small Memphis, U.S.A.\\
\small\tt nathan.lindzey@memphis.edu\\
\and
Levent Tun\c{c}el \thanks{Research of this author
was supported in part by
Discovery Grants from NSERC and by U.S. Office of Naval Research under award numbers
N00014-15-1-2171 and N00014-18-1-2078.}\\
\small Department of Combinatorics and Optimization\\[-0.8ex]
\small University of Waterloo\\[-0.8ex]
\small Waterloo, Canada.\\
\small\tt ltuncel@uwaterloo.ca}

\begin{document}

\maketitle


\begin{abstract}
We explore some connections between association schemes and the analyses of the semidefinite programming (SDP) based convex relaxations of combinatorial optimization problems in the Lov\'{a}sz--Schrijver lift-and-project hierarchy. Our analysis of the relaxations of the stable set polytope leads to bounds on the clique and stability numbers of some regular graphs reminiscent of classical bounds by Delsarte and Hoffman, as well as the notion of deeply vertex-transitive graphs --- highly symmetric graphs that we show arise naturally from some association schemes. We also study relaxations of the hypergraph matching problem, and determine exactly or provide bounds on the lift-and-project ranks of these relaxations. Our proofs for these results also inspire the study of a homogeneous coherent configuration based on hypermatchings, which is an association scheme except it is generally non-commutative. We then illustrate the usefulness of obtaining commutative subschemes from non-commutative homogeneous coherent configurations via contraction in this context.
\end{abstract}

\ignore{
\begin{abstract}
We utilize association schemes to analyze
the quality of semidefinite programming (SDP) based convex relaxations of
integral packing and covering polyhedra determined by matchings in hypergraphs.
As a by-product of our approach, we obtain bounds on the clique and stability
numbers of some regular graphs reminiscent of classical bounds by Delsarte and Hoffman. We determine exactly or provide bounds on the performance of Lov\'{a}sz--Schrijver SDP
hierarchy, and illustrate the usefulness of obtaining commutative subschemes from
non-commutative schemes via contraction in this context.
\end{abstract}
}

\section{Introduction}

Association schemes provide a beautiful unifying framework for algebraic representations of symmetries of
permutation groups. In the space of symmetric $n$-by-$n$ matrices, $\SS^n$, the optimization of a linear function subject to
linear inequalities and equations on the matrix variable together with the positive semidefiniteness constraint, defines
a canonical representation of semidefinite programming (SDP) problems.  Let $\SS^n_+$ denote the set of positive semidefinite
matrices in the set of $n$-by-$n$ symmetric matrices $\SS^n$. Then, the automorphism group of $\SS^n_+$ can be algebraically described as:
\[
\Aut(\SS^n_+) = \left\{A \cdot A^{\top} \,\, : \,\, A \in \R^{n \times n}, \textup{ $A$ is non-singular } \right\}.
\]
That is, an automorphism acts as $X \mapsto A X A^{\top}$.
Clearly, conjugation by any $n$-by-$n$ permutation matrix is in the automorphism group of $\SS^n_+$. So,
if the linear equations and inequalities of our SDP problem are invariant under the action of symmetries of a
permutation group $\cG$, then for every feasible solution $\bar{X}$ of our SDP and for every $\sigma \in \cG$,
we have $\sigma(\bar{X})$ feasible in the SDP (where the second usage of $\sigma$, by our abuse of notation, denotes the action of
the permutation by the conjugation of the underlying permutation matrix, i.e., $\sigma: \SS^n \to \SS^n$).
Therefore, using the convexity of the feasible regions of SDPs, we conclude
\begin{equation}
\label{eqn:sym1}
\frac{1}{|\cG|} \sum_{\sigma \in \cG} \sigma(\bar{X})
\end{equation}
is feasible in our SDP. Similarly, if the objective function of our SDP is also invariant under the action of $\cG$, we conclude
that if $\bar{X}$ is an optimal solution of our SDP then so is the matrix given by \eqref{eqn:sym1}.

In some cases, this invariance group $\cG$ is so rich that the underlying SDP problems can be equivalently written
as linear programming (LP) problems. One of the earliest applications of this idea (in the context of the intersection of
combinatorial optimization, LP and SDP) appears in the 1970s (see \cite{Delsarte73,Lovasz79,Schrijver81}) as well as in the 1990s (see~\cite{Lee90}). More recently,
see \cite{GoemansR99,GatermannP04,deKlerkPS07} and the references therein.

Many, if not most problems in discrete mathematics (in particular, graph theory) are
stated in a way that they already expose potential invariances under certain group actions. In others, we are sometimes able to impose such
symmetries to simplify the analysis. Suitably constructed optimization problems formulating the underlying problem
usually inherit such symmetries. When such optimization problems are intractable, one resorts to
their convex relaxations. These convex relaxations typically inherit and sometimes even further enrich such symmetries.

Thus, tools in association schemes can be extremely helpful in analyzing matrix variables in semidefinite relaxations of which the underlying problem has rich symmetries. Conversely, the analyses of these matrix variables can lead to interesting observations for related association schemes. One of the goals of this paper is to highlight some of these connections between combinatorics and optimization. We aim to introduce the material gently so that the paper will be largely accessible to readers who have some familiarity with association schemes, algebraic graph theory, or semidefinite programming.

In Section~\ref{secPrelim}, we quickly review the basic definitions and facts we need from association schemes. We also introduce the hierarchy of semidefinite programming based convex relaxations generated by the $\LS_+$ operator due to Lov\'{a}sz and Schrijver~\cite{LovaszS91} (also known as the $N_+$ operator in the literature).

In Section~\ref{secStabilityClique}, we provide some bounds on the clique and stability numbers of some regular graphs by utilizing algebraic graph theory techniques and the lift-and-project operator $\LS_+$. The pursuit of when these bounds are tight leads to the notion of deeply vertex-transitive graphs, which are vertex-transitive graphs with additional symmetries and arise rather naturally from some association schemes. One of our main goals in Section~\ref{secStabilityClique} is to introduce our approach in an elementary way and in a well-known setting, where the analysis of a single step of $\LS_+$ hierarchy is relatively simple, and relate our findings to existing results.

In Section~\ref{secLandP}, we delve into our techniques more deeply and analyze the behaviour of $\LS_+$ hierarchy on a variety of integral packing and covering polyhedra arising from matchings in hypergraphs. In the process, we introduce the hypermatching homogeneous coherent configuration, which is almost an association scheme except it is generally non-commutative. Still, the known properties of some of its commutative subschemes are useful in our proofs in this section. 

Finally, in Section~\ref{secHnkl}, we shift our focus from lift-and-project analysis to the aforementioned hypermatching homogeneous coherent configuration, and discuss some of its interesting properties. We also explore the usefulness of obtaining commutative subschemes from non-commutative homogeneous coherent configurations via contracting associates.


\section{Preliminaries}\label{secPrelim}

In this section, we introduce the necessary definitions and notation in association schemes and lift-and-project methods for our subsequent discussion.  We refer the reader to~\cite{BannaiI84, Bailey04, Godsil18, GodsilM15} for a more thorough treatment of association schemes, and to~\cite{AuPhD} for a comprehensive analysis of lift-and-project operators in combinatorial optimization.

\subsection{Association schemes}

Given a set of matrices $\A$, we let $\tn{Span}~\A$ denote the set of matrices that can be expressed as linear combinations of matrices in $\A$. We first define what an \emph{association scheme} is, as well as a slightly more general notion of a \emph{homogeneous coherent configuration}.


\begin{definition}\label{defnscheme}
Let $\Omega$ be a finite set and $\mathcal{I}$ be a set of indices. A set $\mathcal{A} = \{A_i\}_{i \in \mathcal{I}}$ of $|\Omega|$-by-$|\Omega|$ $0,1$-matrices is a \emph{homogeneous coherent configuration} if 
\begin{itemize}
\item[(A1)] $A_i = I$, the identity matrix, for some $i \in \mathcal{I}$,
\item[(A2)] $B \in \mathcal{A} \Rightarrow B^{\top} \in \mathcal{A}$,
\item[(A3)] $\sum_{i \in \mathcal{I}} A_i = J$ where $J$ is the all-ones matrix, and
\item[(A4)] $A_iA_j \in \tn{Span}~\A$ for all $i,j \in \I$.
\end{itemize}
Each non-identity matrix of a homogeneous coherent configuration is referred to as an \emph{associate}.

Furthermore, a homogeneous coherent configuration is an \emph{association scheme} (or simply, a \emph{scheme}) if it further satisfies
\begin{itemize}
\item[(A5)]  $A_iA_j = A_jA_i$ for all $i,j \in \I$.
\end{itemize}
\end{definition}

Thus, a homogeneous coherent configuration is an association scheme except it is not necessarily commutative, and this more general notion will be particularly useful when we discuss homogeneous coherent configurations based on  hypergraph matchings later in the manuscript. For more on homogeneous coherent configuration, the reader may refer to  ~\cite{Cameron01}. Also, to reduce cluttering, we will abbreviate homogeneous coherent configurations as h.c.c.\ from here on. 

When $\mathcal{A}$ is an association scheme, $\tn{Span}~\A$ is a commutative matrix algebra called the \emph{Bose--Mesner algebra} of $\mathcal{A}$. A very useful property of commutative schemes is that the eigenspaces of the matrices in the scheme are aligned. That is, there exists an orthonormal set of eigenvectors $\{v_i\}_{i \in \Omega}$ that are eigenvectors for all matrices in  $\tn{Span}~\A$. 
Therefore, the eigenvalues of any matrix in $\tn{Span}~\A$ can be obtained by taking the corresponding linear combination of the eigenvalues of matrices in $\A$. 

We call a given h.c.c.\ \emph{symmetric} if all of its associates are symmetric matrices. It then follows from property (A4) that if $\A$ is symmetric, then $A_iA_j = A_jA_i$ for all $i,j \in \mathcal{I}$, thus a symmetric h.c.c.\ is also commutative. Not all commutative h.c.c.s are symmetric, but the commutative h.c.c.s that occur in this work will be. 

For any association scheme $\A$ in which each associate has up to $q$ distinct eigenvalues, there is a set of projection matrices $\set{E_i}_{i=1}^q$ where each $E_i$ corresponds to one of the $q$ eigenspaces of matrices in $\A$. Then one can define the \emph{P-matrix} of $\A$ to be the $q$-by-$|\mathcal{I}|$ matrix where $P[i,j]$ is the eigenvalue of $A_j$ corresponding to the projection matrix $E_i$. (Notice that $P$ is necessarily a square matrix, as $|\mathcal{I}| = q$ follows from the fact that the primitive idempotents are a dual basis for the Bose--Mesner algebra.) Characterizing $\set{ E_i}_{i=1}^q$ and $P$ allows us to analyze any matrix $Y$ in the Bose--Mesner algebra of $\A$ in a unified way. This is particularly helpful in SDP problems where all feasible solutions lie in $\tn{Span}~\A$, as it could allow us to reduce the dimension and complexity of the SDP problem significantly, which is usually very helpful both in practice and in theoretical analysis of such SDP problems. In particular, a key consequence is that since the eigenspaces of all associates $A_i$ of $\A$ are aligned, for every matrix $Y \in \tn{Span}~\A$, any cone inequality based on the Loewner order can be rewritten as a set of equivalent linear inequalities. For example, given $B \in \tn{Span}~\A$, define $b \in \mR^{\I}$ such that $B = \sum_{i \in \I} b_i A_i$. Then the constraint $Y \succeq B$ holds if and only if $P(y-b) \geq 0$, where the variable vector $y \in \mR^{\I}$ represents the matrix $Y$ via $Y = \sum_{i \in \I} y_i A_i$.

One of the most ubiquitous association schemes is the Johnson scheme. Let $[p] := \set{1,\ldots ,p}$, and let $[p]_q := \set{S \subseteq [p] : |S| = q}$. Given integers $p,q$, and $i$ where $0 \leq i \leq \min\set{q,p-q}$, define the matrix $J_{p,q,i}$ whose rows and columns are indexed by elements in $[p]_q$, such that
\[
J_{p,q,i}[S,T] := 
\begin{cases}
1 \quad &\text{ if } |S \cap T| = q-i;\\
0 \quad &\text{ otherwise. }
\end{cases}
\]
Notice that $J_{p,q,0}$ is the $\binom{p}{q}$-by-$\binom{p}{q}$ identity matrix, while $J_{p,q,q}$ is the adjacency matrix of the Kneser graph of the $q$-subsets of $[p]$. Given fixed $p$ and $q$, the \emph{Johnson scheme} is the set of matrices $\mathcal{J}_{p,q} := \set{J_{p,q,i} : i=0, \ldots, \min\set{q,p-q}}$. It is easy to check that $\J_{p,q}$ indeed satisfies (A1)-(A4), and is symmetric (and hence commutative). The eigenvalues of the associates in $\J_{p,q}$ are well known (see, for instance,~\cite{Delsarte73, GodsilM15}).

\begin{proposition}\label{JohnsonEvalues}
The eigenvalues of $J_{p,q,i} \in \J_{p,q}$ are
\begin{eqnarray}
\label{JohnsonEvalues1}&& \sum_{h=i}^q (-1)^{h-i+j} \binom{h}{i} \binom{p-2h}{q-h} \binom{ p-h-j}{h-j}\\
\label{JohnsonEvalues2}&=& \sum_{h=0}^q (-1)^h \binom{j}{h} \binom{q-j}{i-h}\binom{p-q-j}{i-h},
\end{eqnarray}
for $j \in \{ 0,1, \ldots, q \}$.
\end{proposition}

Another important association scheme is the Hamming scheme. Given integers $p, q \geq 1$ and $i \in \set{0,1,\ldots, q}$, define the matrix $H_{p,q,i}$ whose rows and columns are indexed by elements in $\set{0,1, \ldots, p-1}^q$, such that
\[
H_{p,q,i}[S,T] := 
\begin{cases}
1 \quad &\tn{if $S,T$ differ at exactly $i$ positions;}\\
0 \quad &\tn{otherwise.}
\end{cases}
\]
Then $H_{p,q,i}$ is a $p^q$-by-$p^q$ matrix, and the \emph{Hamming scheme} $\Hm_{p,q}$ is the set of matrices $\set{H_{p,q,i} : i= 0,\ldots, q}$. As with the Johnson scheme, the Hamming scheme is symmetric and commutative, with well-known eigenvalues (see, for instance,~\cite{BrouwerCIM18}).

\begin{proposition}\label{HammingEvalues}
The eigenvalues of $H_{p,q,i} \in \Hm_{p,q}$ are
\[
 \sum_{h=0}^i (-1)^h (p-1)^{i-h} \binom{j}{r}\binom{q-j}{i-h}
 \]
 for $j \in \set{0, 1, \ldots, q}$.
\end{proposition}

\subsection{Lift-and-project methods and the $\LS_+$ operator}

Before we combine our knowledge of association schemes with the analyses of lift-and-project relaxations, let us first put the lift-and-project approach into perspective and introduce some necessary notation.

When faced with a difficult combinatorial optimization problem, one common approach is to model it as a $0,1$-integer program of the form
\[
\max \set{c^{\top}x : x \in P \cap \set{0,1}^n},
\]
where the set $P \subseteq [0,1]^n$ is convex and tractable (i.e., we can optimize any linear function over it in polynomial time). While integer programs are $\mathcal{NP}$-hard to solve in general, we could discard the integrality constraint, simply optimize $c^{\top}x$ over $P$, and efficiently obtain an approximate solution to the given problem. Furthermore, one can aim to improve upon the initial relaxation $P$. More precisely, given $P \subseteq [0,1]^n$, we define its \emph{integer hull} to be
\[
P_I := \tn{conv}\set{P \cap \set{0,1}^n}.
\]
When $P \neq P_I$, we strive to derive from $P$ another set $P'$ where $P_I \subseteq P' \subset P$. Ideally, this tighter set $P'$ is also tractable, and we can then optimize the same objective function $c^{\top}x$ over $P'$ instead of $P$, and obtain a potentially better approximate solution.

One way to systematically generate such a tighter relaxation is via the \emph{lift-and-project} approach. While there are many known algorithms that fall under this approach (see, among others,~\cite{SheraliA90, BalasCC93, Lasserre01, BienstockZ04, AuT16}), we will focus on the $\LS_+$-operator due to 
Lov{\'a}sz and Schrijver~\cite{LovaszS91}. Given $P \subseteq [0,1]^n$, define the homogenized cone of $P$ to be
\[
K(P) := \set{\begin{bmatrix}\lambda \\ \lambda x \end{bmatrix} \in \mathbb{R}^{n+1}: \lambda \geq 0, x \in P}.
\]
We index the new coordinate by $0$. Also, let $e_i$ be the $i^{\tn{th}}$ unit vector, and recall that $\mS^k$ denotes the set of $k$-by-$k$ symmetric matrices. Then we define
\begin{eqnarray*}
\LS_+(P)&:= &\Big\{x \in \mathbb{R}^n : \exists Y \in \mathbb{S}^{n+1}, \\
&& Ye_0 = \tn{diag}(Y) = \begin{bmatrix}1 \\ x \end{bmatrix}, \\
&&  Ye_i,  Y(e_0 - e_i)\in K(P),~\forall i \in [n], \\
&& Y \succeq 0 \Big\}.
\end{eqnarray*}

Intuitively, the operator $\LS_+$ \emph{lifts} a given $n$-dimensional set $P$ to a collection of $(n+1)$-by-$(n+1)$ matrices, imposes some constraints, and then \emph{projects} the set back down to $\mR^n$. Among other properties, $\LS_+(P)$ satisfies 
\[
P_I \subseteq \LS_+(P) \subseteq P
\]
for every $P \subseteq [0,1]^n$. To see the first containment, note that for every integral vector $x \in P$, $Y := \begin{bmatrix} 1 \\ x \end{bmatrix}\begin{bmatrix} 1 \\ x \end{bmatrix}^{\top}$ satisfies all conditions of $\LS_+$ and thus certifies $x \in \LS_+(P)$. For the second containment, let $x \in \LS_+(P)$ with certificate matrix $Y$ that satisfies $Ye_i, Y(e_0-e_i) \in K(P)$. Then $Ye_0 \in K(P)$ (since the cone $K(P)$ is closed under vector addition), certifying $x \in P$.

Thus, compared to $P$, $\LS_+(P)$ contains exactly the same collection of integer solutions, while providing a tighter relaxation of $P_I$.  Also, if $P$ is tractable, so is $\LS_+(P)$, as optimizing a linear function over this set amounts to solving a semidefinite program whose number of variables and constraints depend polynomially on that of $P$. 

Moreover, $\LS_+$ can be applied iteratively to a set $P$ to obtain yet tighter relaxations. If we let $\LS_+^k(P)$ denote the set obtained from $k$ successive applications of $\LS_+$ to $P$, then it holds in general that
\[
P \supseteq \LS_+(P) \supseteq \LS_+^2(P) \supseteq \cdots \supseteq \LS_+^n(P) = P_I.
\]
Thus, for every set $P \subseteq [0,1]^n$, $\LS_+$ generates a hierarchy of progressively tighter relaxations of $P_I$, with the guarantee that the operator reaches $P_I$ in at most $n$ iterations. For a proof of these properties as well as other aspects of $\LS_+$, the reader may refer to~\cite{LovaszS91}.

Given a set $P \subseteq [0,1]^n$, we define the \emph{$\LS_+$-rank} of $P$ to be the smallest integer $k$ where $\LS_+^k(P) = P_I$. The notion of lift-and-project rank gives us a measure of how far a given relaxation $P$ is from its integer hull $P_I$ with respect to the given lift-and-project operator. In particular, a relaxation having a high $\LS_+$-rank could indicate that the underlying integer hull is difficult to solve for, and/or that the operator $\LS_+$ is not well suited to tackle this particular problem. 

To establish a lower bound on the $\LS_+$-rank of a set, a standard approach is to show that there exists a point $\bar{x} \not\in P_I$ that is contained $\LS_+^k(P)$, which implies that the $\LS_+$-rank of $P$ is at least $k+1$. Verifying $\bar{x} \in \LS_+^k(P)$ would require finding a \emph{certificate matrix} $Y$ that satisfies all conditions specified in the definition of $\LS_+$. As we shall see, this is where symmetries in the given problem can be immensely useful. In particular, the task of establishing the positive semidefiniteness of $Y$ could be significantly simplified by relating it to matrices from association schemes whose eigenvalues are known. For instance, Georgiou~\cite{GeorgiouPhD} used the known eigenvalues of the Johnson scheme when establishing a lower bound on the Lasserre-rank of a relaxation related to the max-cut problem. We shall see a few other examples of this application in this manuscript.

\section{Bounding the clique and stability numbers of graphs}\label{secStabilityClique}

In this section, we focus on the SDP obtained from applying a single iteration of $\LS_+$ to the standard LP relaxation of the stable set problem of graphs. Part of the goal of this section is to introduce the proof techniques that we will need subsequently in this relatively elementary and well-known setting. We will also highlight several points in our discussion that we will revisit in greater depth and complexity in Sections~\ref{secLandP} and~\ref{secHnkl}.
 
The structure of this section is as follows. First, in Section~\ref{subsecStabilityClique1}, we look into the
$\LS_+$-relaxation of the stable set polytope of graphs, and prove our main result of the section (Proposition~\ref{propnka}). Along the way, we introduce the notion of deeply vertex-transitive graphs, which are graphs with very rich symmetries. We then show in Section~\ref{subsecStabilityClique2} how deeply vertex-transitive graphs can be constructed from some association schemes. Finally, in Section~\ref{subsecStabilityClique3}, we relate Proposition~\ref{propnka} to some classic results, such as bounds on the clique and chromatic numbers due to Delsarte~\cite{Delsarte73} and Hoffman~\cite{Hoffman03}.

\subsection{$\LS_+$-relaxations of the stable set polytope and deeply vertex-transitive graphs}\label{subsecStabilityClique1}

Given a simple graph $G$, we define the \emph{fractional stable set polytope} of $G$ to be
\[
\FRAC(G) := \set{ x \in [0,1]^{V(G)} : x_i + x_j \leq 1, ~\forall \set{i,j} \in E(G)}.
\]
We also define the \emph{stable set polytope} to be $\STAB(G) := \FRAC(G)_I$, the convex hull of the integral vectors in $\FRAC(G)$. Notice that a vector $x \in \set{0,1}^{V(G)}$ is contained in $\STAB(G)$ if and only if it is the characteristic vector of a stable set in $G$. We let $\alpha(G)$ denote the stability number of $G$ (i.e., the size of the largest stable set in $G$). We also let $\bar{e}$ be the all-ones vector (of appropriate dimensions). Then we see that
\begin{equation}\label{maxStab1} 
\alpha(G) = \max\set{ \bar{e}^{\top} x : x \in \STAB(G)}.
\end{equation}
Given a graph $G$, we define
\begin{equation}\label{maxStab2} 
\alpha_{\LS_+}(G) := \max\set{ \bar{e}^{\top}x : x \in \LS_+(\FRAC(G))}.
\end{equation}
Since $\STAB(G)  \subseteq  \LS_+(\FRAC(G))$ for every graph $G$, $\alpha(G) \leq \alpha_{\LS_+}(G)$.
Therefore, while it is $\mathcal{NP}$-hard to compute $\alpha(G)$ for a general graph, we can obtain an upper bound on $\alpha(G)$ by solving~\eqref{maxStab2}, which is a semidefinite program of manageable size. It is known that many classical families of inequalities that are valid for $\STAB(G)$ are also valid for $\LS_+(\FRAC(G))$, including (among others) clique, odd cycle, odd antihole, and wheel inequalities~\cite{LovaszS91, LiptakT03}. Moreover, given a graph $G$, consider the theta body of $G$, defined as follows:
\begin{eqnarray*}
\THG(G)&:= &\Big\{x \in \mathbb{R}^n : \exists Y \in \mathbb{S}^{n+1}, \\
&& Ye_0 = \tn{diag}(Y) = \begin{bmatrix}1 \\ x \end{bmatrix}, \\
&&  Y[i,j] = 0, \forall \set{i,j} \in E(G), \\
&& Y \succeq 0. \Big\}
\end{eqnarray*}
(See~\cite{LovaszS91} for a proof, as well as remarks on how the above is equivalent to the conventional definition of the theta body.) From their definitions, it is apparent that $\LS_+(\FRAC(G)) \subseteq \THG(G)$ for all graphs $G$.  Thus, if we define $\theta(G) := \max \set{\bar{e}^{\top}x : x \in \THG(G)}$, it follows that $\alpha_{\LS_+}(G) \leq \theta(G)$ for all graphs.

There has also been recent interest~\cite{BianchiENT13, BianchiENT17, Wagler18, BianchiENW23} in classifying \emph{$\LS_+$-perfect graphs}, which are graphs $G$ where $\LS_+(\FRAC(G)) = \STAB(G)$. Since the stable set polytope of a perfect graph is defined by only clique and non-negative inequalities, every perfect graph is also $\LS_+$-perfect.

For our analysis of $\LS_+(\FRAC(G))$, we are particularly interested in graphs that have plenty of symmetries. Given a vertex $i \in V(G)$, we let $G \ominus i$ denote the subgraph of $G$ induced by vertices that are neither $i$ nor adjacent to $i$. (Equivalently, we obtain $G \ominus i$ from $G$ by removing the closed neighborhood of $i$.) Also, given a graph $G$, we say that a permutation $\sigma$ of $V(G)$ is an \emph{automorphism} if $\set{i,j} \in E(G) \iff \set{ \sigma(i), \sigma(j) } \in E(G)$ for all $i,j \in V(G)$. Then we have the following.

\begin{definition}\label{defnDVT}
A graph $G$ is \emph{deeply vertex-transitive} if
\begin{itemize}
\item[(DVT1)]
$G$ is vertex-transitive;
\item[(DVT2)]
For every $i \in V(G)$ and for every $j_1, j_2 \in V(G \ominus i)$, there exists an automorphism $\sigma$ of $G$ where $\sigma(i) = i$ and $\sigma(j_1) = j_2$.
\item[(DVT3)]
$\overline{G}$, the complement graph of $G$, contains a connected component that is not a complete graph.
\end{itemize}
\end{definition}

Notice that if $G$ is vertex-transitive, then the graphs $\set{G \ominus i : i \in V(G)}$ are all isomorphic. Thus, for $G$ to satisfy (DVT2), it suffices to check the condition for an arbitrary vertex $i$. Also, (DVT3) is equivalent to requiring that the subgraph $G \ominus i$ contain at least one edge for every vertex $i$.

\begin{figure}[ht!]
\begin{center}
\def\x{-360/7}
\def\z{(90+\x)}
\begin{tabular}{cc}
\begin{tikzpicture}
[scale=1.6, thick,main node/.style={circle, minimum size=4mm, inner sep=0.1mm,draw,font=\tiny\sffamily}]
\node[main node] at ({cos((-1)*\x+\z)},{sin((-1)*\x+\z)}) (1) {$1$};
\node[main node] at ({cos(0*\x+\z)},{sin((0)*\x+\z)}) (2) {$2$};
\node[main node] at ({cos(1*\x+\z)},{sin((1)*\x+\z)}) (3) {$3$};
\node[main node] at ({cos((2)*\x+\z)},{sin((2)*\x+\z)}) (4) {$4$};
\node[main node] at ({cos(3*\x+\z)},{sin((3)*\x+\z)}) (5) {$5$};
\node[main node] at ({cos(4*\x+\z)},{sin((4)*\x+\z)}) (6) {$6$};
\node[main node] at ({cos((5)*\x+\z)},{sin((5)*\x+\z)}) (7) {$7$};

  \path[every node/.style={font=\sffamily}]
(1) edge (2)
(2) edge (3)
(3) edge (4)
(4) edge (5)
(5) edge (6)
(6) edge (7)
(7) edge (1);
\end{tikzpicture}
&
\begin{tikzpicture}
[scale=1.6, thick,main node/.style={circle, minimum size=4mm, inner sep=0.1mm,draw,font=\tiny\sffamily}]
\node[main node] at ({cos((-1)*\x+\z)},{sin((-1)*\x+\z)}) (1) {$1$};
\node[main node] at ({cos(0*\x+\z)},{sin((0)*\x+\z)}) (2) {$2$};
\node[main node] at ({cos(1*\x+\z)},{sin((1)*\x+\z)}) (3) {$3$};
\node[main node] at ({cos((2)*\x+\z)},{sin((2)*\x+\z)}) (4) {$4$};
\node[main node] at ({cos(3*\x+\z)},{sin((3)*\x+\z)}) (5) {$5$};
\node[main node] at ({cos(4*\x+\z)},{sin((4)*\x+\z)}) (6) {$6$};
\node[main node] at ({cos((5)*\x+\z)},{sin((5)*\x+\z)}) (7) {$7$};

  \path[every node/.style={font=\sffamily}]
(1) edge (3)
(1) edge (4)
(1) edge (5)
(1) edge (6)
(2) edge (4)
(2) edge (5)
(2) edge (6)
(2) edge (7)
(3) edge (5)
(3) edge (6)
(3) edge (7)
(4) edge (6)
(4) edge (7)
(5) edge (7);
\end{tikzpicture}
\end{tabular}
\caption{The $7$-cycle (left) and the $7$-antihole (right)}\label{figDVT1}
\end{center}
\end{figure}

\begin{example}\label{egDVT1}
Let $G$ be the $7$-cycle, as shown in Figure~\ref{figDVT1} (left). Then while $G$ is vertex-transitive, it is not deeply vertex-transitive, as there does not exist an automorphism $\sigma$ where $\sigma(1) = 1$ and $\sigma(3) = 4$, and so $G$ does not satisfy (DVT2).

On the other hand, consider $\overline{G}$, the 7-antihole (Figure~\ref{figDVT1}, right). It is indeed vertex-transitive and its complement (i.e., the 7-cycle) is connected and not a complete graph, and so both (DVT1) and (DVT3) hold. Next, observe that $\overline{G} \ominus 1$ is a complete graph formed by the two vertices $2$ and $7$, and the following automorphism $\sigma$ shows that (DVT2) holds as well.
\[
\begin{array}{r|c|c|c|c|c|c|c}
i & 1 & 2 & 3 & 4 & 5 & 6 & 7\\
\hline
\sigma(i) & 1 & 7 & 6 & 5 & 4 & 3 & 2
\end{array}
\]
In fact, one can extend this argument to show that the complement of the $n$-cycle is deeply vertex-transitive for every $n \geq 5$.
\end{example}

\begin{figure}[ht!]
\begin{center}
\def\x{-360/5}
\def\z{(90+\x)}
\def\y{0.55}
\begin{tabular}{cc}
\begin{tikzpicture}
[scale=2, thick,main node/.style={circle, minimum size=4mm, inner sep=0.1mm,draw,font=\tiny\sffamily}]
\node[main node] at ({cos((-1)*\x+\z)},{sin((-1)*\x+\z)}) (1) {$1$};
\node[main node] at ({cos(0*\x+\z)},{sin((0)*\x+\z)}) (2) {$2$};
\node[main node] at ({cos(1*\x+\z)},{sin((1)*\x+\z)}) (3) {$3$};
\node[main node] at ({cos((2)*\x+\z)},{sin((2)*\x+\z)}) (4) {$4$};
\node[main node] at ({cos(3*\x+\z)},{sin((3)*\x+\z)}) (5) {$5$};
\node[main node] at ({cos((-1)*\x+\z)*\y},{sin((-1)*\x+\z)*\y}) (6) {$6$};
\node[main node] at ({cos((0)*\x+\z)*\y},{sin((0)*\x+\z)*\y}) (7) {$7$};
\node[main node] at ({cos((1)*\x+\z)*\y},{sin((1)*\x+\z)*\y}) (8) {$8$};
\node[main node] at ({cos((2)*\x+\z)*\y},{sin((2)*\x+\z)*\y}) (9) {$9$};
\node[main node] at ({cos((3)*\x+\z)*\y},{sin((3)*\x+\z)*\y}) (10) {$10$};

  \path[every node/.style={font=\sffamily}]
(1) edge (2)
(2) edge (3)
(3) edge (4)
(4) edge (5)
(5) edge (1)
(1) edge (6)
(2) edge (7)
(3) edge (8)
(4) edge (9)
(5) edge (10)
(8) edge (6)
(9) edge (7)
(10) edge (8)
(6) edge (9)
(7) edge (10);
\end{tikzpicture}
&

\begin{tikzpicture}
[scale=2, thick,main node/.style={circle, minimum size=4mm, inner sep=0.1mm,draw,font=\tiny\sffamily}]
\node[main node] at ({cos((-1)*\x+\z)},{sin((-1)*\x+\z)}) (1) {$1$};
\node[main node] at ({cos(0*\x+\z)},{sin((0)*\x+\z)}) (2) {$5$};
\node[main node] at ({cos(1*\x+\z)},{sin((1)*\x+\z)}) (3) {$4$};
\node[main node] at ({cos((2)*\x+\z)},{sin((2)*\x+\z)}) (4) {$9$};
\node[main node] at ({cos(3*\x+\z)},{sin((3)*\x+\z)}) (5) {$6$};
\node[main node] at ({cos((-1)*\x+\z)*\y},{sin((-1)*\x+\z)*\y}) (6) {$2$};
\node[main node] at ({cos((0)*\x+\z)*\y},{sin((0)*\x+\z)*\y}) (7) {$10$};
\node[main node] at ({cos((1)*\x+\z)*\y},{sin((1)*\x+\z)*\y}) (8) {$3$};
\node[main node] at ({cos((2)*\x+\z)*\y},{sin((2)*\x+\z)*\y}) (9) {$7$};
\node[main node] at ({cos((3)*\x+\z)*\y},{sin((3)*\x+\z)*\y}) (10) {$8$};

  \path[every node/.style={font=\sffamily}]
(1) edge (2)
(2) edge (3)
(3) edge (4)
(4) edge (5)
(5) edge (1)
(1) edge (6)
(2) edge (7)
(3) edge (8)
(4) edge (9)
(5) edge (10)
(8) edge (6)
(9) edge (7)
(10) edge (8)
(6) edge (9)
(7) edge (10);
\end{tikzpicture}
\end{tabular}
\caption{The Petersen graph (left) and its image (right) under the automorphism $\sigma$ from Example~\ref{egDVT2}}\label{figDVT2}
\end{center}
\end{figure}

\begin{example}\label{egDVT2}
Let $G$ be the Petersen graph, as shown in Figure~\ref{figDVT2} (left). Since $G$ is indeed vertex-transitive and its complement is a connected and non-complete graph, both (DVT1) and (DVT3) hold. Next, consider the permutation $\sigma$ where
\[
\begin{array}{r|c|c|c|c|c|c|c|c|c|c}
i & 1 & 2 & 3 & 4 & 5 & 6 & 7 & 8 & 9 & 10\\
\hline
\sigma(i) & 1 & 5 & 4 & 9 & 6 & 2 & 10 & 3 & 7 & 8 
\end{array}
\]
In Figure~\ref{figDVT2} (right), we replaced every vertex label $i$ by $\sigma(i)$ to help visualize that $\sigma$ is indeed an automorphism. Notice that $\sigma$ fixes vertex $1$ while rotating the six vertices in the $6$-cycle in $G \ominus 1$. Thus, $G$ also satisfies (DVT2), and is indeed deeply vertex-transitive. In fact, the same automorphism $\sigma$ can also be used to show that the complement of the Petersen graph is also deeply vertex-transitive. We will revisit this example from a different perspective in Example~\ref{egDVT4}.
\end{example}

\begin{figure}[ht!]
\begin{center}
\def\x{-360/8}
\def\z{(90+\x)}
\begin{tikzpicture}
[scale=1.6, thick,main node/.style={circle, minimum size=4mm, inner sep=0.1mm,draw,font=\tiny\sffamily}]
\node[main node] at ({cos((-1)*\x+\z)},{sin((-1)*\x+\z)}) (1) {$1$};
\node[main node] at ({cos(0*\x+\z)},{sin((0)*\x+\z)}) (2) {$2$};
\node[main node] at ({cos(1*\x+\z)},{sin((1)*\x+\z)}) (3) {$3$};
\node[main node] at ({cos((2)*\x+\z)},{sin((2)*\x+\z)}) (4) {$4$};
\node[main node] at ({cos(3*\x+\z)},{sin((3)*\x+\z)}) (5) {$5$};
\node[main node] at ({cos(4*\x+\z)},{sin((4)*\x+\z)}) (6) {$6$};
\node[main node] at ({cos((5)*\x+\z)},{sin((5)*\x+\z)}) (7) {$7$};
\node[main node] at ({cos((6)*\x+\z)},{sin((6)*\x+\z)}) (8) {$8$};

  \path[every node/.style={font=\sffamily}]
(1) edge (2)
(1) edge (3)
(2) edge (3)
(2) edge (4)
(3) edge (4)
(3) edge (5)
(4) edge (5)
(4) edge (6)
(5) edge (6)
(5) edge (7)
(6) edge (7)
(6) edge (8)
(7) edge (8)
(7) edge (1)
(8) edge (1)
(8) edge (2);
\end{tikzpicture}
\caption{A graph $G$ where $G \ominus i$ is vertex-transitive for every $i$, while $G$ is not deeply vertex-transitive}\label{figDVT3}
\end{center}
\end{figure}

\begin{example}\label{egDVT3}
Given a deeply vertex-transitive graph $G$ and $i \in V(G)$, (DVT2) implies that $G \ominus i$ must be a vertex-transitive graph in its own right. However, we remark that the converse is not true. Let $G$ be the graph in Figure~\ref{figDVT3}. Then while $G \ominus 1$ (which is a clique on three vertices) is vertex-transitive, there is no automorphism $\sigma$ where $\sigma(1) = 1$ and $\sigma(4) = 5$.

Another such example is the (5-regular) Clebsch graph. While destroying any vertex yields the Petersen graph (which is vertex-transitive), the Clebsch graph itself does not satisfy (DVT2), and thus is not deeply vertex-transitive.
\end{example}

Next, given a graph $G$ that is $k$-regular with $n \geq 4$ vertices, we define the quantity
\[
\Delta(G) := 
\begin{cases}
 \frac{2n-3k}{n-3} & \tn{if $G$ contains a triangle;}\\
 \frac{n-3k+2 + \sqrt{ (n-3k+2)^2 + 4(n-3)(n-2k)}}{2(n-3)} & \tn{otherwise.}
\end{cases}
\]
Notice that, whenever $n \geq 4$ and $k \geq 1$,
\begin{eqnarray*}
&& \frac{n-3k+2 + \sqrt{ (n-3k+2)^2 + 4(n-3)(n-2k)}}{2(n-3)}\\
 &\leq & \frac{ n-3k+2 + \sqrt{(n-3k+2)^2} + \sqrt{4(n-3)(n-2k)}}{2(n-3)} \\
&\leq & \frac{n-3k+2 + (n-3k+2) + 2(n-k-\frac{3}{2})}{2(n-3)}\\
& \leq& \frac{2n-3k}{n-3}.
\end{eqnarray*}
Thus, $\Delta(G) \leq \frac{2n-3k}{n-3} \leq 2$ for all regular graphs $G$ on $4$ or more vertices. The following is the main result of this section.

\begin{proposition}\label{propnka}
Suppose $G$ is a $k$-regular graph on $n \geq 4$ vertices, and let $\lambda_2$ be the second largest eigenvalue of $G$. Define $a := \max\set{1, \lambda_2, \Delta(G)}$. Then
\begin{equation}\label{propnkaineq}
\alpha_{\LS_+}(G) \geq \frac{n-k+a}{a+1}.
\end{equation}
Moreover, if $G$ is deeply vertex-transitive, then equality holds in~\eqref{propnkaineq}.
\end{proposition}

\begin{proof}
We first prove the inequality for all $k$-regular graphs. For convenience, let $\b_1 := \frac{2n-3k}{n-3}$ and $\b_2 :=  \frac{n-3k+2 + \sqrt{ (n-3k+2)^2 + 4(n-3)(n-2k)}}{2(n-3)}$ throughout this proof. First, we will show that the given bound applies for $a := \max\set{1, \ld_2, \b_1}$ for all graphs, and then mention how, in the case of triangle-free graphs, we might be able to use a smaller $a$ and potentially obtain a better lower bound for $\a_{\LS_+}(G)$.

Let $A(\overline{G})$ be the adjacency matrix of the complement of $G$. Note that the largest and smallest eigenvalues of $A(\overline{G})$ are $n-k-1$ and $-1-\ld_2$, respectively. Also, let $d := \frac{n(a+1)^2}{n-k+a}$, and we show that the certificate matrix $Y := \dfrac{1}{d} \begin{bmatrix} d & (a+1)\bar{e}^{\top} \\ (a+1)\bar{e} & (a+1)I + A(\overline{G})\end{bmatrix}$ satisfies all conditions of $\LS_+$. First, for every $i \in V(G)$,
\[
(Ye_i)[j] = \begin{cases}
\frac{a+1}{d} & \tn{if $j=0$ or $j=i$;}\\
0 & \tn{if $j \in V(G)$ is adjacent to $i$;}\\
\frac{1}{d} & \tn{otherwise.}
\end{cases}
\]
Since $a\geq 1$, the above vector must belong to $K(\FRAC(G))$. Likewise, notice that $Ye_i \geq \begin{bmatrix} \frac{a+1}{d} \\ 0 \end{bmatrix}$ for all $i$. Hence, 
\[
Ye_0 - Ye_i \leq \begin{bmatrix} 1 \\ \frac{a+1}{d} \bar{e} \end{bmatrix} - \begin{bmatrix} \frac{a+1}{d} \\  0 \end{bmatrix} = \frac{1}{d}   \begin{bmatrix} d-(a+1) \\ (a+1) \bar{e} \end{bmatrix}.
\]
Now, observe that $a \geq \frac{2n-3k}{n-3} =\b_1$ if and only if $d \geq 3(a+1)$. Thus,
\[
Ye_0 - Ye_i = \frac{1}{d}   \begin{bmatrix} d-(a+1) \\ (a+1) \bar{e} \end{bmatrix} \leq \frac{d-(a+1)}{d}  \begin{bmatrix}1 \\ \frac{1}{2}\bar{e} \end{bmatrix},
\]
and so $Ye_0 - Ye_i \in K(\FRAC(G))$ for every $i \in V(G)$.

Finally, notice that the minimum eigenvalue of $(a+1)I + A(\overline{G})$ is $(a+1) + (-1- \lambda_2)$, which is nonnegative since $a \geq \ld_2$. Thus, using the Schur complement, we see that $Y \succeq 0$ if and only if the eigenvalue of 
\[
(a+1)I + A(\overline{G}) - \frac{1}{d}(a+1)\bar{e} (a+1) \bar{e}^{\top} = (a+1)I + A(\overline{G}) - \frac{n-k-a}{n}J
\]
corresponding to $\bar{e}$ is non-negative. Indeed, one can check that this eigenvalue is $2a$, which is positive.

Since all conditions of $\LS_+$ are met, we conclude that $\frac{a+1}{d} \bar{e} \in \LS_+(G)$. Therefore, 
\[
\alpha_{\LS_+}(G) \geq \bar{e}^{\top} \left( \frac{a+1}{d} \bar{e} \right) = \frac{n-k+a}{a+1}
\]
for $a := \max\set{1, \ld_2, \b_1}$ for all $k$-regular graphs.

We now consider the case where $G$ does not contain a triangle. Let us take a closer look at the vector $Ye_0 - Ye_i$. Observe that
\[
(Ye_0 - Ye_i)[j] = \begin{cases}
1- \frac{a+1}{d} & \tn{if $j=0$;}\\
0 & \tn{if $j=i$;}\\
\frac{a+1}{d} & \tn{if $j \in V(G)$ is adjacent to $i$;}\\
\frac{a}{d} & \tn{otherwise.}
\end{cases}
\]
Since $G$ does not have a triangle, there are no edges $\set{j_1, j_2} \in E(G)$ where both $j_1, j_2$ are adjacent to $i$. Thus, to ensure $Ye_0 - Ye_i \in K(\FRAC(G))$, it suffices to have
\[
\frac{a+1}{d} + \frac{a}{d} \leq 1- \frac{a+1}{d},
\]
or equivalently $d \geq 3a+2$. This holds if $a \geq \frac{n-3k+2 + \sqrt{(n-3k+2)^2 + 4(n-3)(n+2k)}}{2(n-3)} = \b_2$. This proves the bound for graphs without triangles.

We next prove the reverse inequality for deeply vertex-transitive graphs, and again first establish the case where $G$ contains a triangle. Notice that if $G$ is deeply vertex-transitive, then $\ld_2 > -1$. To see that, let $\ld_1 \geq \ld_2 \geq \cdots \geq \ld_n$ be the eigenvalues of $G$. If $\ld_2 \leq -1$, then $\sum_{i=2}^n \ld_i \leq (n-1)\ld_2 \leq -(n-1)$. On the other hand, it is true for all graphs that $\ld_1 \leq n-1$ and the $n$ eigenvalues sum to zero. This implies that $\ld_1 = n-1$ and $\ld_2 = \cdots = \ld_n = -1$ in this case. Thus $G$ is a complete graph, which violates (DVT3) and thus is not deeply vertex-transitive. Hence, we will assume that $\ld_2 > -1$ for the rest of the proof. 

Next, let $\bar{x} \in \LS_+(\FRAC(G))$ where $\bar{e}^{\top} \bar{x} = \alpha_{\LS_+}(G)$. Then there must be a matrix $\bar{M}$ such that $Y = \begin{bmatrix} 1 & \bar{x}^{\top} \\ \bar{x} & \bar{M} \end{bmatrix}$ satisfies all conditions imposed by $\LS_+$. Now, let $P_{\sigma}$ be a permutation matrix corresponding to $\sigma \in \Aut(G)$, the automorphism group of $G$. Consider the matrix 
\[
Y' :=\frac{1}{|\Aut(G)|} \sum_{\sigma \in \Aut(G)} \begin{bmatrix} 1 & (P_{\sigma} \bar{x})^{\top} \\
P_{\sigma} \bar{x} & P_{\sigma} \bar{M} P_{\sigma}^{\top}
\end{bmatrix}.
\]
Since $\sigma \in \Aut(G)$, every matrix in the sum above also satisfies all conditions for $\LS_+$. This implies that $Y'$ is also a certificate matrix (since the set of certificate matrices is a convex set), and that $Y'e_0 \in K(\LS_+(\FRAC(G)))$. Also, since $G$ is vertex-transitive, it must be the case that $Ye_0 = \begin{pmatrix} 1 \\ b\bar{e} \end{pmatrix}$ for some constant $b$. Moreover, since $\bar{e}^{\top}\bar{x} = \bar{e} ^{\top}(P_{\sigma}\bar{x})$ for every permutation $\sigma$, $\bar{e}^{\top} (b \bar{e}) = \bar{e}^{\top} \bar{x} = \a_{\LS_+}(G)$.

Furthermore, by (DVT2) we see that for every $i \in V(G)$ and $j_1, j_2 \in V(G \ominus i)$, there exists $\sigma \in \Aut(G)$ where $\sigma(i) = i$ and $\sigma(j_1) = j_2$. Hence, it follows that $Y'[i, j_1] = Y'[i, j_2]$ for all $j_1, j_2 \in V(G \ominus i)$. Since $Y'$ is symmetric, we deduce that $Y'[i,j]$ must be constant over all distinct $i,j \in V(G)$ where $\set{i,j} \not\in E(G)$. Therefore, we see that $Y' = \begin{bmatrix} 1 & b\bar{e}^{\top} \\ b\bar{e} & bI + c A(\overline{G}) \end{bmatrix}$ for some real numbers $b$ and $c$. Now consider some of the restrictions $\LS_+$ imposes on $Y'$ (and hence $b$ and $c$). First, $ bI + c A(\overline{G}) \succeq 0$ implies that
\begin{equation}\label{propnka3}
 c \leq \frac{b}{\lambda_2 +1}.
 \end{equation}
 Note that we applied the assumption $\ld_2 > -1$ here. Likewise, $bI + c A(\overline{G}) - (b\bar{e})(b\bar{e}^{\top}) \succeq 0$ implies
 \begin{equation}\label{propnka5}  
 b + c(n-k-1) - b^2n \geq 0.
 \end{equation}
 Next, since $\overline{G}$ contains a component that is not the complete graph, there exists vertex $i$ in this component that is adjacent to vertices $j_1, j_2$ where $\set{j_1, j_2} \not\in E(\overline{G})$. This means that the subgraph $(G \ominus i)$ contains at least one edge $\set{j_1,j_2}$, and so the condition $Ye_i \in K(\FRAC(G))$ imposes that $Ye_i[j_1] + Ye_i[j_2] \leq Ye_i[0]$, which implies that
 \begin{equation}\label{propnka4}
c \leq \frac{b}{2}.
 \end{equation}
 Likewise, since $G$ is vertex-transitive and contains a triangle by assumption, for every vertex $i \in V(G)$ there exist two vertices $j_1, j_2$ that are adjacent to $i$ while $\set{j_1, j_2} \in E(G)$. Then since $(Ye_0-Ye_i)[j_1] = (Ye_0-Ye_i)[j_2] = b$ and $(Ye_0-Ye_i)[0] = 1-b$, the constraint
\[
(Ye_0-Ye_i)[j_1] + (Ye_0-Ye_i)[j_2] \leq (Ye_0-Ye_i)[0]
\]
implies
 \begin{equation}\label{propnka4a}
b \leq \frac{1}{3}.
 \end{equation}
Now, if we define $a' := \max\set{ \ld_2, 1}$, then~\eqref{propnka3} and~\eqref{propnka4} hold if and only if $c \leq \frac{b}{a'+1}$. Combining this with~\eqref{propnka4a} and~\eqref{propnka5} yields $b \leq \min\set{ \frac{1}{n} \left(\frac{n-k+a'}{a'+1}\right), \frac{1}{3}}$. Since $\frac{1}{n} \left(\frac{n-k+a'}{a'+1}\right) \leq \frac{1}{3}$ if and only if $a' \geq \b_1(G)$, this finishes the proof for graphs that contain a triangle.

Finally, if $G$ does not contain a triangle, then instead of imposing $3b \leq 1$ as in~\eqref{propnka4a}, the condition $Ye_0 - Ye_i \in K(\FRAC(G))$ would impose the weaker inequality $3b-c \leq 1$. Combined with~\eqref{propnka3},~\eqref{propnka4}, and~\eqref{propnka5}, we obtain that $b \leq \frac{n-k+a'}{a'+1}$ where $a' \geq \max\set{ 1, \ld_2, \b_2}$. This finishes the proof.
\end{proof}

Given a regular graph $G$, if we know that $\alpha(G) < \frac{n-k+a}{a+1}$, then Proposition~\ref{propnka} implies that $\alpha(G) < \alpha_{\LS_+}(G)$, and consequently $\LS_+(\FRAC(G)) \neq \STAB(G)$. On the other hand, given a deeply vertex-transitive graph $G$, Proposition~\ref{propnka} determines $\alpha_{\LS_+}(G)$ and implies that $\alpha(G) \leq \lfloor \frac{ n-k+a}{a+1} \rfloor$. We also remark that the same ingredients used in the proof of Proposition~\ref{propnka} can be used to prove a similar (but weaker) bound for $\theta(G)$:

\begin{proposition}\label{propnkatheta}
Suppose $G$ is a $k$-regular, non-complete graph on $n$ vertices, and let $\lambda_2$ be the second largest eigenvalue of $G$. Then
\begin{equation}\label{propnkathetaineq}
\theta(G) \geq \frac{n-k+\ld_2}{\ld_2+1}.
\end{equation}
Moreover, if $G$ is deeply vertex-transitive, then equality holds in~\eqref{propnkathetaineq}.
\end{proposition}

Thus, it follows from Propositions~\ref{propnka} and~\ref{propnkatheta} that $\alpha_{\LS_+}(G) < \theta(G)$ in some situations when $\ld_2 < 1$, as we shall see in Example~\ref{example1}.

\begin{example}\label{example1}
Let $G$ be the odd antihole with $n= 2\ell+1$ vertices, for some integer $\ell \geq 2$.
Then, as discussed in Example~\ref{egDVT1}, $G$ is deeply vertex-transitive. 
In this case, $\ld_2 = -1 + 2\cos\left( \frac{\pi}{2\ell+1}  \right) < 1$, and $\Delta(G) \leq 1$ for all $\ell \geq 2$. This implies that $\theta(G) =1+  \sec\left( \frac{\pi}{2\ell+1}\right) > 2$, while $\alpha_{\LS_+}(G) = 2 = \alpha(G)$.
\end{example}

Next, we provide another example showing that the bound in Proposition~\ref{propnka} can indeed be not tight when $G$ is not deeply vertex-transitive.

\begin{example}\label{example2}
Let $G$ be the $7$-cycle (which obviously does not contain a triangle). Then $G$ is 2-regular with $\ld_2 = 2\cos\left(\frac{2\pi}{7}\right) \approx 1.247$, and $\Delta(G) = \frac{3+ \sqrt{57}}{8} \approx 1.319$. Thus, $a = \frac{3+ \sqrt{57}}{8}$, and Proposition~\ref{propnka} implies that $\alpha_{\LS_+}(G) \geq 2.72$, which is not tight as $\alpha(G) = \alpha_{\LS_+}(G) = 3$.

In this case, while it is true that $\frac{3}{7} \bar{e} \in \LS_+(\FRAC(G))$, we cannot use the symmetry reduction in the proof of Proposition~\ref{propnka} to deduce that $\frac{3}{7} \bar{e}$ has a certificate matrix with constant non-edge entries, due to the fact that  $G \ominus i$ is not vertex-transitive for any vertex $i$. In fact, one can check that the unique matrix that certifies $\frac{3}{7}\bar{e} \in \LS_+(\FRAC(G))$ is
\[
Y = \frac{1}{7} \begin{bmatrix}
7 & 3 & 3& 3 & 3& 3 & 3& 3 \\
3 & 3 & 0 & 2 & 1 & 1 & 2 & 0 \\
3& 0  & 3 & 0 & 2 & 1 & 1 & 2 \\
3& 2 & 0  & 3 & 0 & 2 & 1 & 1  \\
3& 1 & 2 & 0  & 3 & 0 & 2 & 1 \\
3& 1 & 1 & 2 & 0  & 3 & 0 & 2\\
3& 2  &1 & 1 & 2 & 0  & 3 & 0 \\
3& 0 & 2  & 1 & 1 & 2 & 0  & 3 \\
\end{bmatrix}.
\]
\end{example}

We also remark that deeply vertex-transitivity is incomparable with arc-transitivity, a well-studied property that also characterizes graphs with rich symmetries (see, for instance,~\cite[Chapter 4]{GodsilR13}). For example, odd antiholes are deeply vertex-transitive but not arc-transitive (or even edge-transitive), while the opposite is true for odd cycles of length at least $7$.

\subsection{Constructing deeply vertex-transitive graphs from association schemes}\label{subsecStabilityClique2}

Next, we describe how deeply vertex-transitive graphs can arise from some association schemes. Given a graph $G$, let $A(G)$ denote the adjacency matrix of $G$, and conversely given a symmetric $0,1$-matrix $A$ we let $G(A)$ denote the undirected graph whose adjacency matrix is $A$.

We first prove a result that will help show that some graphs related to the Johnson scheme are deeply vertex-transitive.

\begin{lemma}\label{lemJohnsonDVT}
Let $p,q,i$ be integers where $p \geq q \geq i$. If $G(J_{p,q,i})$ contains a connected component that is not the complete graph, then $\overline{G(J_{p,q,i})}$ is deeply vertex-transitive.
\end{lemma}

\begin{proof}
For convenience, let $G := \overline{G(J_{p,q,i})}$ throughout this proof. Also, given a permutation $\sigma$ on $[p]$, we extend the notation and define
\begin{equation}\label{lemJohnsonDVTeq1}
\sigma(S) := \set{ \sigma(s) : s \in S}
\end{equation}
for every $S \subseteq [p]$. We can think of~\eqref{lemJohnsonDVTeq1} as extending $\sigma$ from a bijection from $[p]$ to $[p]$ to a bijection from $[p]_q$ to $[p]_q$ for every $q$ where $0 \leq q \leq p$.

Now, to verify (DVT1), let $S,T \in V(G)$. Observe that as long as the permutation $\sigma : [n] \to [n]$ satisfies $\sigma(S) = T$, then $\sigma$ defines an automorphism of $G$ that maps vertex $S$ to vertex $T$. Likewise, for (DVT2), let $S \in V(G)$ and $T_1, T_2 \in V(G \ominus S)$. Then we know that $|T_1 \cap S| = |T_2 \cap S| = q-i$. This implies that $|S \setminus T_1| = |S \setminus T_2| = i$ and $|T_1 \setminus S| = |T_2 \setminus S| = i$. Thus, there exists a bijection $\sigma : [p] \to [p]$ where $\sigma( T_1 \cap S) = T_2 \cap S, \sigma(S \setminus T_1) = S \setminus T_2$, and $\sigma(T_1 \setminus S) = T_2 \setminus S$. Then notice that $\sigma$ gives an automorphism of $G$ where $\sigma(S) = S$ and $\sigma(T_1) = T_2$. Thus, (DVT2) holds.

Therefore, if $\overline{G} = G(J_{p,q,i})$ contains a non-complete component, then (DVT3) would hold as well, and our claim follows.
\end{proof}

\begin{example}\label{egDVT4}
Consider the graph $G := G(J_{p,2,1})$ for some $p \geq 5$. Then $\overline{G} = G(J_{p,2,2})$. Since both $G$ and $\overline{G}$ are connected and non-complete graphs, we obtain from Lemma~\ref{lemJohnsonDVT} a family of graphs $G$ where both $G$ and $\overline{G}$ are deeply vertex-transitive.

Also, since $G(J_{5,2,2})$ gives the Petersen graph, we obtain an alternative proof that both the Petersen graph and its complement are deeply vertex-transitive, as shown earlier in Example~\ref{egDVT2}.
\end{example}

As we shall see later in the proof of Proposition~\ref{propnkaHamming}, we can use similar ideas from the proof of Lemma~\ref{lemJohnsonDVT} to show that the complements of some graphs that arise from the Hamming scheme are also deeply vertex-transitive. Next, with Lemma~\ref{lemJohnsonDVT}, we can use Proposition~\ref{propnka} to verify if optimizing over $\LS_+(\FRAC(G))$ gives the correct stability number of some deeply vertex-transitive graphs related to the Johnson scheme.

\begin{proposition}\label{propnkaJohnson}
\begin{itemize}
\item[(i)]
Given integers $q \geq 2$ and $p \geq q+2$, let $G := \overline{G(J_{p,q,1})}$. Then
\[
\alpha_{\LS_+}(G) = p-q+1 = \alpha(G).
\]
\item[(ii)]
Given integers $q \geq 2$ and $p \geq 2q+1$, let $G := \overline{G(J_{p,q,q})}$. Then
\[
\alpha_{\LS_+}(G) = \frac{p}{q}, 
\]
which is equal to $\alpha(G) = \lfloor \frac{p}{q} \rfloor$ if and only if $q | p$.
\end{itemize}
\end{proposition}

\begin{proof}
For (i), we see that $G$ has $n= \binom{p}{q}$ vertices, each with degree $k = \binom{p}{q} -q(p-q)-1$. Also, with $q \geq 2$ and $p \geq q+2$, $\overline{G}$ indeed has a non-complete component, and thus $G$ is deeply vertex-transitive by Lemma~\ref{lemJohnsonDVT}. From Proposition~\ref{JohnsonEvalues}, one obtains the minimum eigenvalue of $J_{p,q,1}$ is $-q$, occurring when $j=q$. Thus, the second largest eigenvalue of $G$ is $\ld_2 = q-1 \geq 1$. Also, observe that
\[
\Delta(G) \leq \frac{2n-3k}{n-3} = \frac{3q(p-q)-3 - \binom{p}{q}}{\binom{p}{q}-3} \leq \frac{3(\frac{p}{2})(p-\frac{p}{2}) - 3 - \binom{p}{2}}{\binom{p}{2}-3} \leq 1
\]
for all $p \geq 4$. Thus, $a = \ld_2$, and we obtain from Proposition~\ref{propnka} that
\[
\alpha_{\LS_+}(G) = \frac{n-k+\ld_2}{\ld_2+1} = p-q+1.
\]
Since $\alpha(G) \leq \alpha_{\LS_+}(G)$ in general, it only remains to show that $\alpha(G) \geq p-q+1$. If we let $S_j := [q-1] \cup \set{j}$, then it is easy to check that $S_q, S_{q+1}, \ldots, S_{p}$ form a stable set in $G$ as any two of these sets have $q-1$ elements in common. Thus, (i) follows.

The proof of (ii) is similar. In this case, $n= \binom{p}{q}, k= \binom{p}{q} - \binom{p-q}{q} - 1$, and $\ld_2 = \binom{p-q-1}{q-1}-1$. Notice that $\ld_2 \geq 2$ for all $p\geq 4$ and $q \geq 2$, while $\Delta(G) \leq 2$  for all graphs $G$. Also $G(J_{p,q,q})$ (which is a Kneser graph) is connected and non-complete when $q \geq 2$ and $p \geq 2q+1$. Thus, $G$ is deeply vertex-transitive and $a = \ld_2$, so Proposition~\ref{propnka} implies that 
\[
\alpha_{\LS_+}(G) =\frac{n-k+\ld_2}{\ld_2+1} = \frac{p}{q}.
\]
On the other hand, a stable set in $G$ corresponds to a collection of disjoint $q$-subsets of $[p]$, and so $\alpha(G) = \lfloor \frac{p}{q} \rfloor$. This finishes our proof.
\end{proof}

Proposition~\ref{propnkaJohnson}(ii) implies that when $q$ does not divide $p$, $\FRAC(\overline{G(J_{p,q,q})})$ has $\LS_+$-rank at least $2$. We will revisit these polytopes from a different perspective and determine their exact $\LS_+$-rank when we study matchings in hypergraphs in Section~\ref{secLandP}.

\subsection{Relating Proposition~\ref{propnka} to classic bounds on clique and chromatic numbers}\label{subsecStabilityClique3}

Next, we relate Proposition~\ref{propnka} to some well-known bounds on the clique and chromatic numbers of graphs. For ease of comparison, let us restate Proposition~\ref{propnka} in terms of the clique number of a graph. If we let $\omega(G)$ be the size of the largest clique in $G$, then Proposition~\ref{propnka} readily implies the following:

\begin{corollary}\label{propnkaclique}
Suppose $G$ is a graph whose complement $\overline{G}$ is deeply vertex-transitive. If we let $\ld_1$ and $ \ld_n$ be the maximum and minimum eigenvalues of $G$ respectively, then
\[
\omega(G) \leq \left\lfloor 1 - \frac{\ld_1}{\min\set{\ld_n, -2, -1-\Delta(\overline{G})}} \right\rfloor.
\]
\end{corollary}

\begin{proof}
Let $\ld_1 \geq \ld_2 \geq \cdots \geq \ld_n$ be the eigenvalues of $G$, and likewise $\overline{\ld}_1 \geq \overline{\ld}_2 \geq \cdots \geq \overline{\ld}_n$ be the eigenvalues of $\overline{G}$. Observe that $\ld_1 = n - \overline{\ld}_1 - 1$ and $\ld_n = -1 - \overline{\ld}_2$. Also, let $\overline{a} := \max\set{1, \overline{\ld}_2, \Delta(\overline{G})}$ and $a' := \min\set{ \ld_n, -2, -1-\Delta(\overline{G})}$, then $a' = -1 - \overline{a}$ regardless of the values of $\overline{\ld}_2$ and $\Delta(\overline{G})$. Next,  applying Proposition~\ref{propnka} to $\overline{G}$, we obtain
\[
\alpha(\overline{G}) \leq \alpha_{\LS_+}(\overline{G}) = \frac{n - \overline{\ld}_1 + \overline{a}}{\overline{a}+1} = 1 - \frac{\ld_1}{a'}.
\]
Since it is obvious that $\omega(G)$ is an integer and is equal to $\alpha(\overline{G})$ for every graph, the claim follows.
\end{proof}

We next relate Corollary~\ref{propnkaclique} to some well-known results. First, the following is due to Delsarte~\cite{Delsarte73}, which establishes a similar upper bound on the clique number for a different family of graphs.

\begin{proposition}\label{Delsarteclique}
Let $G$ be a graph whose adjacency matrix $A$ is an associate in an association scheme. Then
\[
\omega(G) \leq \left\lfloor 1 - \frac{\ld_1}{\ld_n} \right\rfloor.
\]
\end{proposition}

Thus, Corollary~\ref{propnkaclique} could provide a tighter upper bound than Delsarte's in cases where $\ld_n > \min\set{ -2, -1 - \Delta(\overline{G})} $. (See, for instance,~\cite{CvetkovicRS04} for more on the rather restrictive families of graphs where $\ld_n > -2$.) Corollary~\ref{propnkaclique} also covers some graphs whose adjacency matrix does not belong to an association scheme.

In addition to Delsarte's result, Hoffman~\cite{Hoffman03} has a similar bound on $\chi(G)$, the chromatic number of a graph.

\begin{proposition}\label{Hoffmanchromatic}
Let $G$ be a graph. Then
\[
\chi(G) \geq \left\lceil 1 - \frac{\ld_1}{\ld_n} \right\rceil.
\]
\end{proposition}

Since $\chi(G) = \omega(G)$ for perfect graphs, combining Hoffman's bound and Corollary~\ref{propnkaclique} implies that if $G$ is perfect and deeply vertex-transitive, then $\ld_n \leq -2, \omega(G) = 1 - \frac{\ld_1}{\ld_n}$, and that $\ld_n$ must divide $\ld_1$. More recently, Godsil et al.~\cite[Lemma 5.2]{GodsilRSS16} showed that  $\theta(\overline{G}) = 1 - \frac{\ld_1}{\ld_n}$ if $G$ is $1$-homogeneous, thus implying that $\omega(G) \leq \lfloor 1 - \frac{\ld_1}{\ld_n}\rfloor$ for these graphs. $1$-homogeneous graphs contain graphs that are both vertex-transitive and edge-transitive, and can be shown to be incomparable with deeply vertex-transitive graphs using the same odd antihole and odd cycle examples mentioned earlier. 

Finally, we conclude this section by considering another example that highlights an idea we will discuss further in Sections~\ref{secLandP} and~\ref{secHnkl}. Given an integer $\ell \geq 2$, consider the Hamming scheme $\Hm_{2, 2\ell+1}$, and define the graph 
\[
G_{\ell} := G(H_{2, 2\ell+1, \ell} + H_{2, 2\ell+1, \ell+1}).
\]
In other words, $G_{\ell}$ has vertex set $\set{0,1}^{2\ell + 1}$, and two vertices are joined by an edge if their corresponding binary strings differ by $\ell$ or $\ell+1$ positions. Then we have the following:

\begin{proposition}\label{propnkaHamming}
For every $\ell \geq 2$, 
\begin{equation}\label{HammingBound}
\omega(G_{\ell}) \leq 2 \ell + 2.
\end{equation}
\end{proposition}

\begin{proof}
We provide two proofs to this claim. First, we claim that $\overline{G_{\ell}}$ is deeply vertex-transitive. While (DVT1) and (DVT3) are relatively straightforward to check, we provide the details to verifying (DVT2). First, we associate each vertex in $\overline{G_{\ell}}$ (a binary string of length $2\ell+1$) with the set $S \subseteq [2\ell+1]$ which contains the position of the $1$'s in the string. Now let $S := \es$ (i.e., the string of all zeros), and let $T_1, T_2 \in V ( \overline{G_{\ell}} \ominus S)$. Then $|T_1|, |T_2| \in \set{\ell, \ell+1}$. We now construct an automorphism of $\overline{G}_{\ell}$ that maps $\es$ to $\es$ and $T_1$ to $T_2$.

If $|T_1| = |T_2|$, let $\sigma: [2\ell+1] \to [2\ell+1]$ be a permutation where $\sigma(T_1) = T_2$. Then $\sigma$ (extended to a permutation on the power set $2^{[2\ell+1]}$)
gives an automorphism of $\overline{G_{\ell}}$ that also satisfies $\sigma(\es) = \es$. Next, if $|T_1| \neq |T_2|$, then $|T_1| + |T_2| = 2\ell +1$. Let $\sigma : [2\ell+1] \to [2\ell+1]$ be a permutation where $\sigma(T_1) = [2\ell +1] \setminus T_2$. Then $\sigma' : 2^{[2\ell+1]} \to 2^{[2\ell+1]}$ where
\[
\sigma'(T) := [2\ell+1] \setminus \sigma(T)
\]
is an automorphism of $\overline{G_{\ell}}$ that satisfies $\sigma'(T_1) = T_2$. The only issue is that $\sigma'(\es) = [2\ell+1]$ and $\sigma'([2\ell+1]) = \es$. However, since $\es$ and $[2\ell+1]$ are adjacent to exactly the same set of vertices in $\overline{G}_{\ell}$, if we define $\sigma'' : 2^{[2\ell+1]} \to 2^{[2\ell+1]}$ where
\[
\sigma''(T) := \begin{cases}
\sigma'(T) & \tn{if $T \not\in \set{\es, [2\ell+1]}$;}\\
T & \tn{if $T \in \set{\es, [2\ell+1]}$,}
\end{cases}
\]
then $\sigma''$ is also an automorphism of $\overline{G_{\ell}}$, with $\sigma''(\es) = \es$ and $\sigma''(T_1) = T_2$. This shows that $\overline{G_{\ell}}$ satisfies (DVT2).

Since $\overline{G_{\ell}}$ is deeply vertex-transitive, Corollary~\ref{propnkaclique} applies. Observe that $G_{\ell}$ is $2 \binom{2\ell+1}{\ell}$-regular, so $\ld_1 = 2 \binom{2\ell+1}{\ell}$. For the least eigenvalue, it follows from~\cite[Proposition 2.2]{BrouwerCIM18} that $\ld_n = \frac{-4}{\ell+1} \binom{2\ell -1}{ \ell}$ (notice that $\ld_n \leq -2$ for all $\ell \geq 2$). Also, one can check that $\Delta(\overline{G_{\ell}}) < 1$ for all $\ell \geq 2$. Thus, we obtain that
\[
\omega(G_{\ell}) \leq 1 - \frac{2 \binom{\ell+1}{\ell}}{ \frac{-4}{\ell+1} \binom{2\ell-1}{\ell}} = 2\ell + 2.
\]
Next, we present another proof that uses Delsarte's bound (Proposition~\ref{Delsarteclique}). While $G_{\ell}$ is not the graph of a single associate in the Hamming scheme, we can define an alternative association scheme of which $A(G_{\ell})$ is an associate. For each $j \in [\ell]$, we define
\[
B_j := H_{2, 2\ell+1, 2\ell+1-j} + H_{2, 2\ell+1, j}.
\]
Then one can check that
\[
\Hm' := \set{I, B_1, B_2, \ldots, B_{\ell}}
\]
is indeed a symmetric h.c.c.\ (and thus an association scheme). We mention the details of verifying (A4) (as (A1) to (A3) are relatively straightforward to check). Let $P := H_{2, 2\ell+1, 2\ell+1}$ for convenience. Notice that $P$ is a permutation matrix that satisfies $P^2 = I$, and that 
\[
H_{2, 2\ell+1, 2\ell+1 - j} = PH_{2, 2\ell+1, j} = H_{2, 2\ell+1, j}P
\]
for every $j \in \set{0, \ldots, 2\ell+1}$. Now, given $B_i, B_j \in \Hm'$,
\[
B_iB_j = (P+I)H_{2, 2\ell+1, i} (P+I)H_{2, 2\ell+1, j} = (P+I)H_{2, 2\ell+1, i}H_{2, 2\ell+1, j}.
\]
Since $H_{2, 2\ell+1, i}H_{2, 2\ell+1, j} \in \tn{Span}~\Hm_{2, 2\ell+1}$, and that $B_i = (P+I)H_{2, 2\ell+1, i}$ for every $i \in [\ell]$, it follows that $(P+I)H_{2, 2\ell+1, i}H_{2, 2\ell+1, j} \in \tn{Span}~\Hm'$.

Finally, since $B_{\ell}$ is the adjacency matrix of the graph $G_{\ell}$, Proposition~\ref{Delsarteclique} applies and we again obtain the bound~\eqref{HammingBound}.
\end{proof}

When is the bound in~\eqref{HammingBound} tight? First, we see that when $\ell$ is odd, \eqref{HammingBound} is tight if and only if there exists a $(2\ell+2)$-by-$(2\ell+2)$ Hadamard matrix. Given a $(2\ell+2)$-by-$(2\ell+2)$ Hadamard matrix $H$, removing one row from $H$ yields $2\ell+2$ column vectors that are binary, with pairwise distance $\ell$ or $\ell+1$. Conversely, given a clique $C$ in $G_{\ell}$ of size $2\l+2$, we may assume (since $G_{\ell}$ is vertex-transitive) that $C = \set{ 0 } \cup S_{\ell} \cup S_{\ell+1}$, where every vector in $S_{\l}$ has $\l$ ones, and every vector in $S_{\ell+1}$ has $\ell+1$ ones. Then we define the set of vectors $C' \subseteq \set{0,1}^{2\ell+2}$ where
\[
C' = \set{ 0 } \cup
\set{ \begin{bmatrix} v \\ 1 \end{bmatrix}  : v \in S_{\ell}} \cup
\set{ \begin{bmatrix} v \\ 0 \end{bmatrix}  : v \in S_{\ell+1}}.
\]
Since $\ell$ is odd, the vectors in $C'$ have pairwise distance $\l+1$, and one can construct a Hadamard matrix from $C'$. This shows that the bound in Proposition~\ref{propnkaHamming} is tight for infinitely many values of $\ell$. On the other hand, the bound is not tight for $\ell=2$ as one can check that $\omega(G_2) = 5$. It would be interesting to determine the values of $\ell$ for which the bound in~\eqref{HammingBound} is tight.

Also, the proof of Proposition~\ref{propnkaHamming} demonstrates the usefulness of finding commutative subschemes within a given h.c.c.\ Given a h.c.c.\ $\A$, we say that a set of matrices $\A'$ is a \emph{(commutative) subscheme} of $\A$ if every matrix in $\A'$ is the sum of a subset of matrices (not necessarily plural) in $\A$, and that $\A'$ is an association scheme in its own right. Thus, in the proof above, $\Hm'$ is a commutative subscheme of $\Hm_{2, 2\ell+1}$. Given an association scheme, it is easy to check if a certain contraction of its associates lead to a subscheme (see, for instance,~\cite[Section 4.2]{Godsil18}). We shall see in the next section that the situation is more complicated when we have an h.c.c.\ that is not necessarily commutative.

\section{Lift-and-project ranks for hypermatching polytopes}\label{secLandP}

In this section, we study $\LS_+$-relaxations related to matchings in hypergraphs. More elaboratively, given a $q$-uniform hypergraph $G$ and an integer $r \geq 1$, let $E_{r}(G)$ denote the set of matchings in $G$ of size $r$. That is, $S \in E_{r}(G)$ if $S = \set{S_1, \ldots, S_r}$ where $S_1, \ldots, S_r \in E(G)$ and are mutually disjoint. Consider the following optimization problem: Given a graph $G$, what is the maximum number of disjoint $r$-matchings in $G$ such that their union is also a matching in $G$? Notice that $E_{1}(G) = E(G)$ and so when $r=1$ this problem reduces to the classical matching problem of finding the largest subset of hyperedges that are mutually disjoint. Given a set of hyperedges $S$ and vertex $i$, we also say that $S$ \emph{saturates} $i$ if $i$ is contained in at least one hyperedge in $S$. Next, we define the polytope
\[
\MT_{r}(G) := \set{ x \in [0,1]^{E_{r}(G)} : \sum_{\substack{S \in E_{r}(G) \\ S~\tn{saturates}~i}} x_S \leq 1,~\forall i \in V(G)}.
\]
Then each integral vector in $\MT_{r}(G)$ corresponds to a set of $r$-matchings in $G$ where no vertex is saturated by more than one matching in this set. 

Consider $G := K_p^q$, the complete $q$-uniform hypergraph on $p$ vertices. That is, $G$ is the graph where $V(G) := [p]$ and $E(G) := [p]_q$. In this case, since each $r$-matching saturates $qr$ vertices, it is apparent that one can choose up to $\lfloor \frac{p}{qr} \rfloor$ disjoint $r$-matchings. Thus, we obtain that
\[
\max \set{\bar{e}^{\top}x : x \in \MT_r(G)_I} = \left\lfloor \frac{p}{qr}\right\rfloor.
\]
Next, we compute the optimal value of the linear program
\begin{equation}\label{MTLP1}
\max \set{ \bar{e}^{\top} x : x \in \MT_{r}(G)}.
\end{equation}
Let $[p]_q^r$ denote $E_r(K_p^q)$ for convenience. Notice that
\[
\left|[p]_q^r \right| = \frac{1}{r!}\binom{p}{q}\binom{p-q}{q}\cdots \binom{p-(r-1)q}{q} = \frac{p!}{r! (q!)^r (p-qr)!}.
\]
Also, every fixed vertex in $[p]$ is saturated by exactly $
\binom{p-1}{qr-1} |[qr]_q^r| $ distinct $r$-matchings. Thus, we see that the optimal value of~\eqref{MTLP1} is attained by the solution
\[
\bar{x} := \left(\binom{p-1}{qr-1} |[qr]_q^r|  \right)^{-1} \bar{e},
\]
giving an optimal value of
\begin{equation}\label{MTLP1a}
\bar{e}^{\top}\bar{x} = \frac{ |[p]_q^r|}{\binom{p-1}{qr-1} |[qr]_q^r| } = \frac{p}{qr}.
\end{equation}
Therefore, $\MT_{r}(G) \neq \MT_{r}(G)_I$ when $qr$ does not divide $p$, and one could apply $\LS_+$ to $\MT_{r}(G)$ to obtain better relaxations of $\MT_{r}(G)_I$. This leads naturally to the question of determining the $\LS_+$-rank of $\MT_{r}(G)$ when $p$ is not a multiple of $qr$. 

For the case $q=2$ and $r=1$, the given problem reduces to finding a maximum matching in ordinary graphs, which is well known to be solvable in polynomial time~\cite{Edmonds65}. Strikingly, it was shown~\cite{StephenT99} that for every positive integer $p$, the $\LS_+$-rank of $\MT_1(K_{2p+1})$ is $p$, providing what was then the first known family of instances where $\LS_+$ requires exponential effort to return the integer hull of a given set. In the lower-bound analysis therein, the authors explicitly described the eigenvalues and eigenvectors of their certificate matrix, which is closely related to matrices in $\tn{Span}~\J_{2p+1,2}$ (also see, \cite{BianchiPhD}).

Herein, we generalize their result to $r$-matchings in hypergraphs. Our main result of this section is the following.

\begin{theorem}\label{STgen1}
Given positive integers $p,q,r$ where $p >qr$ and $qr$ does not divide $p$, the $\LS_+$-rank of $\MT_{r}(K_{p}^q)$ is $\lfloor \frac{p}{qr} \rfloor$.
\end{theorem}

We do note that, if we let $\nu(G)$ be the size of the largest matching in a graph $G$, then the size of the largest $r$-matching is simply $\lfloor \nu(G) / r \rfloor$. Thus, the maximum $r$-matching problem may seem to be an unnecessary generalization of the maximum matching problem. However, as we shall see, the introduction of the additional parameter $r$ allows us to work with a very rich h.c.c.\ based on matchings in hypergraphs, and analyzing the $\LS_+$-certificate matrices in this more general contexts make the tools presented in our proofs more readily translatable to future analyses of other semidefinite relaxations.

This section is structured as follows: We first introduce the hypergraph matching h.c.c.\ in Section~\ref{sec30}. While this scheme is not commutative in general, we will point out a number of its commutative subschemes that are easier to work with.  We then provide the proof of Theorem~\ref{STgen1} in Section~\ref{sec31} while pointing out some immediate consequences of the result. Finally, in Section~\ref{sec32} we analyze the $\LS_+$-rank of the $b$-matching polytope, which gives another example of using known eigenvalues from familiar association schemes to help analyze lift-and-project relaxations.

\subsection{The hypergraph matching homogeneous coherent configuration}\label{sec30}

Consider the complete hypergraph $G = K_p^q$. Recall that, given a permutation $\sigma : [p] \to [p]$ and a set $W \subseteq [p]$, we let $\sigma(W) := \set{ \sigma(j) : j \in W}$ for convenience. Next, we define the equivalence relation on $[p]_q^r \times [p]_q^r$ as follows:

\begin{definition}\label{defEquivRel}
Given $S,S',T,T' \in [p]_q^r$, define the relation $\sim$ where $(S,T) \sim (S',T')$ if there exists a permutation $\sigma : [p] \to [p]$ such that
\begin{itemize}
\item[(I1)] for every hyperedge $S_i \in S$, the hyperedge $\sigma(S_i) \in S'$, and
\item[(I2)] for every hyperedge $T_i \in T$, the hyperedge $\sigma(T_i) \in T'$.
\end{itemize}
\end{definition}

For instance, under this relation, there are $10$ equivalence classes in the cases where $q=2, r=2$, and $p \geq 8$ as illustrated in Figure~\ref{fig1}. Now let $X_0, \ldots, X_m \subseteq [p]_q^r \times [p]_q^r$ denote these equivalence classes, and define $|[p]_q^r|$-by-$|[p]_q^r|$ matrices $M_{p,q,r,i}$ where
\[
M_{p,q,r,i}[S,T] := 
\begin{cases}
1 \quad &\text{ if $(S,T) \in X_i$,}\\
0 \quad &\text{ otherwise. }
\end{cases}
\]
for every $i \in \{0,1, \ldots, m\}$. Then the \emph{$(p,q,r)$-hypermatching h.c.c.}\ is defined to be $\M_{p,q,r} := \set{ M_{p,q,r,i}}_{i=0}^m$. To see that this is indeed an h.c.c., consider the action of the symmetric group $\S_p$ on a matching $T = \set{ T_i }_{i=1}^r \in [p]_q^r$ defined such that, given $\sigma \in \S_p$,
\begin{equation}\label{groupaction1}
\sigma \cdot T = \set{ \sigma(T_i) }_{i=1}^{r}.
\end{equation}
We can further extend this action to pairs of $r$-matchings by defining that, given $S, T \in [p]_q^r$,
\begin{equation}\label{groupaction2}
\sigma \cdot (S,T)  = \left( \sigma \cdot S, \sigma \cdot T \right).
\end{equation}
Observe that there is a one-to-one correspondence between the orbits of this action on $[p]_q^r \times [p]_q^r$ and the aforementioned equivalence classes. In particular, the elements in the orbit associated with the isomorphism class $X_i$ are precisely the indices of the non-zero entries of the matrix $M_{p,q,r,i}$. Thus, it follows that $\M_{p,q,r}$ is a \emph{Schurian coherent configuration} (see~\cite{Cameron01}, for example), which assures that the properties (A1), (A3), and (A4) hold. Furthermore, notice that the group action defined above is transitive on $[p]_q^r$ (and thus only has one orbit), and therefore (A2) holds as well. Hence, $\M_{p,q,r}$ is indeed an h.c.c.

While we are unaware of previous literature that studies the h.c.c.\ $\M_{p,q,r}$ in its full generality (which is not commutative in general), there are some notable choices of $p,q,r$ where $\M_{p,q,r}$ specializes to familiar and commutative schemes. For instance, when $r=1$, each matching has exactly one hyperedge and thus can simply be seen as a $q$-subset of $[p]$, and so $\M_{p,q,1} = \J_{p,q}$ for all $p$ and $q$. Another case where $\M_{p,q,r}$ reduces to the Johnson scheme is when $q=1$, where we obtain $\M_{p,1,r} = \J_{p,r}$. In the cases when $p=qr$ (i.e., each matching is a partition of the $p$ vertices into $r$ subsets of size $q$), Godsil and Meagher~\cite{GodsilM10} showed that $\M_{qr,q,r}$ is a commutative scheme if and only if $q=2$, or $r=2$, or $(q,r) \in \set{ (3,3), (3,4), (4,3), (5,3)}$. They also showed that $\M_{p,q,r}$ is commutative when $q=2$ and $p=2r+1$, as well as when $r=2$ and $p=2q+1$.

As mentioned previously, the commutativity of a scheme $\A$ is a very desirable property that allows us to have a much better handle on the eigenvalues of the matrices in $\tn{Span}~\A$. Hence, given a non-commutative h.c.c., it can be helpful to instead work with subschemes of it that are commutative. While there are various notions of subschemes in the existing literature (see, for instance, \cite{Bailey04, Godsil18}), we will focus on obtaining subschemes by \emph{contraction}. More precisely, given an h.c.c.\ $\mathcal{A}$ and $S \subseteq \mathcal{A}$ a subset of the associates, we define the contraction of $S$ in $\A$ to be the collection of matrices:
\[
\A' := ( \A  \sm S) \cup \set{ \sum_{A_i \in S} A_i}.
\]
That is, we remove matrices in $S$ from $\A$ and replace them by a single matrix that is the sum of all matrices in $S$. If $\A'$ satisfies the properties (A1)-(A5) and thus is an association scheme in its own right, then we call $\A'$ a \emph{(commutative) subscheme} of $\A$. Note that the trivial scheme with just one associate (i.e., $\A = \set{I, J-I}$) is a commutative subscheme of every non-trivial scheme defined on the same ground set.

While $\M_{p,q,r}$ is not necessarily commutative, we point out that it must have at least one non-trivial commutative subscheme as long as $p > qr$. Observe that, given two matchings $S,T \in [p]_q^r$, $S \cup T$ saturates at least $qr$ and up to $2qr$ vertices. Now, for every $i \in \set{qr, qr+ 1, \ldots, 2qr}$, define the matrix $B_{i}\in \mR^{[p]_q^r \times [p]_q^r}$ such that 
\[
B_{i}[S,T] = \left\{
\begin{array}{ll}
1 &\tn{if $S \cup T$ saturates $i$ vertices;}\\
0 & \tn{otherwise.}
\end{array}
\right.
\]
Also, given a pair of symmetric association schemes $\A_1 = \{A^{(1)}_i\}_{i \in \mathcal{I}_1}$ and $\A_2 = \{A^{(2)}_i\}_{i \in \mathcal{I}_2}$ on ground sets $\Omega_1$ and $\Omega_2$ respectively, we define their \emph{wreath product} $\A_1 \wr \A_2$ to be the set consisting of the matrices:
\[ (A^{(1)}_i \otimes J_{|\Omega_2|}) \text{ for all $i\in\mathcal{I}_1$ where $A_i \neq I$} \quad \text{ and } \quad (I_{|\Omega_1|} \otimes A^{(2)}_i) \text{ for all } i \in \mathcal{I}_2. \]
It is known that the wreath product of two symmetric schemes must also be a symmetric scheme (see, for example,~\cite{Bailey04} for a proof). Then we have the following:

\begin{proposition}\label{HtildeHoverline}
Given positive integers $p,q,r$ where $p > qr$,
\begin{itemize}
\item[(i)]
The set of matrices
\[
\tilde{\M}_{p,q,r} = \set{I, B_{qr}-I} \cup \set{B_i}_{i=qr+1}^{\min\set{p, 2qr}}
\]
is a commutative subscheme of $\M_{p,q,r}$.
\item[(ii)]
If $\M_{qr, q,r}$ is a symmetric association scheme, let $\set{ M_1, \ldots, M_d} \subset \M_{p,q,r}$ be the associates corresponding to isomorphism classes where the union of two matchings saturates exactly $qr$ vertices. (Notice it then follows that $B_{qr} = I + \sum_{i=1}^d M_i$.) Then
\[
\overline{\M}_{p,q,r} = \set{I} \cup \set{M_i}_{i=1}^d \cup \set{B_i}_{i=qr+1}^{\min\set{p, 2qr}}
\]
is a commutative subscheme of $\M_{p,q,r}$.
\end{itemize}
\end{proposition}

\begin{proof}
(i) We show that $\tilde{\M}_{p,q,r}$ is the wreath product of two simple schemes. Let $\mathcal{K} = \set{I, J-I}$ be the trivial association scheme defined on the ground set $[qr]_q^r$. Then we claim that
\[
\tilde{M}_{p,q,r} \cong \J_{p, qr} \wr \mathcal{K} =  \set{ J_{p,qr, i} \otimes J}_{i=1}^{qr} \cup \set{ I, I \otimes (J - I)}.
\]
To see this, first notice that each element of the ground set of the scheme $\J_{p, qr}$ is a subset of $[p]$ of size $qr$. Given such a set $W \in [p]_{qr}$ with the elements of $W$ being $w_1, w_2, \ldots, w_{qr}$ listed in ascending order, we can consider $W$ as the function from $[qr]$ to $[p]$ where $W(i) = w_i$ for every $i$. Also, given any matching $S \in [qr]_q^r$, each ordered pair $(W,S)$ naturally corresponds to a matching in  $[p]_q^r$, obtained from applying $W$ to all vertices in $S$ in the same way as described in \eqref{groupaction1}.

Now, let $W,W' \in [p]_{qr}$ and $S, S' \in [qr]_{q}^r$. Consider the two matchings $T = (W,S), T' =(W',S') \in [p]_{q}^r $. Notice that for every integer $i, qr \leq i \leq \min\set{p, qr}$,
\[
(J_{p,qr, i-qr} \otimes J)[T,T'] = 1
\]
if and only if $J_{p, qr, i-qr}(W,W') =1$ (i.e., $|W \cup W'| = i$) and $J[S,S'] = 1$ (which is true since $J$ is the matrix of all ones). This happens if and only if $T \cup T'$ saturates exactly $i$ vertices, which is the case exactly when $B_{i}[T,T']=1$.

Next, observe that
\[
( I \otimes (J - I))[T,T'] = 1
\]
if and only if $I[W,W']=1$ (i.e., $W=W'$), and $(J-I)[S,S'] = 1$ (i.e., $S \neq S'$). This is equivalent to saying that $T,T'$ are distinct matchings that saturate exactly the same $qr$ vertices, and hence $(B_{qr} - I)[T,T'] =1$. This proves that the scheme $\tilde{\M}_{p,q,r}$ is equivalent to $\J_{p,qr} \wr \K$.

(ii) In the case when $\M_{qr,q,r}$ is a symmetric association scheme itself, one can show that $\overline{\M}_{p,q,r} = \J_{p, qr} \wr \M_{qr,q,r}$ using essentially the same argument for (i), which implies that $\overline{\M}_{p.q.r}$ is a commutative subscheme of $\M_{p,q,r}$.
\end{proof}

The above result shows that when $p > qr$ and $\M_{qr,q,r}$ is a symmetric and non-trivial scheme (e.g., when $q=2$ and $r \geq 3$), then there are at least two distinct commutative subschemes of $\M_{p,q,r}$. In Section~\ref{secHnkl}, we will return to the question of which contractions of $\M_{p,q,r}$ lead to commutative subschemes.

For now, we will focus on the subscheme $\tilde{\M}_{p,q,r}$. Since it is simply the wreath product of the Johnson scheme and the trivial scheme, we can easily obtain the eigenvalues of any matrix $M \in \tn{Span}~\tilde{\M}_{p,q,r}$ as long as we can express $M$ as a linear combination of matrices in $\tilde{\M}_{p,q,r}$, which is very easy if we know the entries of $M$. This will be useful in our analyses of lift-and-project relaxations subsequently in this section.

On the other hand, while $\tn{Span}~\overline{\M}_{p,q,r}$ gives a broader set of matrices than $\tn{Span}~\tilde{\M}_{p,q,r}$ and still possesses the aligned-eigenspaces property, we have less of a grip on the eigenvalues of the matrices therein as the eigenvalues for the associates in $\M_{qr, q,r}$ are less well understood, even in the cases when it is indeed a commutative scheme.

\subsection{Packing matchings in hypergraphs}\label{sec31}

Having introduced the hypermatching h.c.c.\ $\M_{p,q,r}$ and discussed some of its commutative subschemes, we are now ready to prove Theorem~\ref{STgen1}. Again, the case where $q=2$ and $r=1$ was first shown in~\cite{StephenT99}. Our proof uses many similar ideas as theirs, as well as the knowledge of the eigenvalues of matrices in $\tilde{\M}_{p,q,r}$.


\begin{proof}[Proof of Theorem~\ref{STgen1}]
For convenience, let $G := K_{p}^q$ and  $P := \MT_{r}(G)$ throughout this proof. We first prove the lower bound of the rank. Let $\a_0 := \left( \binom{p-1}{qr-1} |[qr]_q^r| \right)^{-1}$. Notice that when $qr < p < 2qr$, $\a_0 \bar{e} \in P \sm P_I$ (as explained in~\eqref{MTLP1a} when we computed the optimal value of~\eqref{MTLP1}) and so $\LS_+$-rank of $P$ is at least $1$. Thus, for the rest of our lower-bound argument, we may assume that $p \geq 2qr$.

Next, we show that $\a_0 \bar{e}  \in \LS_+^{\ell} (P)$ for all integers $\ell < \lfloor \frac{p}{qr} \rfloor$. Note that $\a_0 \bar{e} \not\in P_I$ if $qr$ does not divide $p$, so the claim above would imply that $P$ has $\LS_+$-rank at least $\lfloor \frac{p}{qr} \rfloor$.

We prove our claim by induction on $\ell$. The base case $\ell=0$ is immediate as $\a_0 \bar{e} \in \LS_+^0(P) = P$ for all $p,q$, and $r$. 

For the inductive step, let $\a_1 := \left( \binom{p-qr-1}{qr-1}|[qr]_q^r| \right)^{-1}$, and so $\a_1 \bar{e}  \in \LS_+^{\ell-1} \left(\MT_{r}(K_{p-qr}^q) \right)$ by the inductive hypothesis. Also, given a set of edges $C \subseteq E(G)$, we let $\sat(C) \subseteq V(G)$ be the set of vertices  saturated by the set $C$. Define the certificate matrix  
\[
Y := \begin{bmatrix}
1 & \a_0 \bar{e}^{\top} \\
\a_0 \bar{e}& Y'
\end{bmatrix}
\]
where the $|[p]_q^r|$-by-$|[p]_q^r|$ matrix $Y'$  has entries
\[
Y'[S,T] = 
\begin{cases}
\a_0 & \tn{if $S=T$;}\\
\a_0\a_1 & \tn{if $\sat(S) \cap \sat(T) = \es$;}\\
0 & \tn{otherwise.}
\end{cases}
\]
We show that $Y$ satisfies all conditions imposed by $\LS_+$. First, it is apparent that $Y$ is symmetric, and $Ye_0 =  \tn{diag}(Y) = \begin{bmatrix} 1 \\ \a_0 \bar{e} \end{bmatrix}$.

Next, given $S \in [p]_q^r$, define the set
\[
F := \set{x \in [0,1]^{[p]_q^r} : x_S = 1, x_T = 0~\tn{for all $T$ where $\sat(S) \cap \sat(T) \neq \es$}}.
\]
By the inductive hypothesis $\a_1 \bar{e} \in \LS_+^{\ell-1}\left(\MT_{r}(K_{p-qr}^q) \right)$. Moreover, observe that the projection of $P$ onto the coordinates not restricted to $0$ or $1$ in $F$ is exactly $\MT_{r}(K_{p-qr}^q)$. Since $\LS_+$ satisfies the general property that $\LS_+(P \cap F) \subseteq \LS_+(P) \cap F$ for every face $F$ of the unit hypercube, if we define vector $w_S \in  \mR^{[p]_q^r}$ where
\[
w_S[T] = 
\begin{cases}
1 & \tn{if $S= T$;}\\
\a_1 & \tn{if $\sat(S) \cap \sat(T) = \es$;}\\
0 & \tn{otherwise,}
\end{cases}
\]
it then follows that $w_S \in \LS_+^{\ell-1}(P)$. Thus, $Ye_S = \a_0 \begin{bmatrix}1 \\ w_S\end{bmatrix} \in K\left( \LS_+^{\ell-1}(P) \right)$. Next, we show that $Y(e_0 - e_S) \in K\left( \LS_+^{\ell-1}(P) \right)$. We claim that, for every matching $S \in [p]_q^r$,
\begin{equation}\label{STgen1eq1}
\sum_{T \in [p]_q^r} \frac{|\sat(S) \cap \sat(T)|}{qr} Ye_T = Ye_0.
\end{equation}
Notice that the coefficient of $Ye_S$ in the left hand side of~\eqref{STgen1eq1} is $\frac{|\sat(S) \cap \sat(S)|}{qr} = \frac{qr}{qr} = 1$. Thus, the above implies that $Y(e_0-e_S)$ can be expressed as a non-negative linear combination of vectors in $\set{ Ye_T : T \in [p]_q^r}$. Since we have shown that $Ye_T$ is contained in the convex cone $K\left( \LS_+^{\ell-1}(P) \right)$ for every $T \in [p]_q^r$, it follows from~\eqref{STgen1eq1} that $Y(e_0 - e_S) \in K\left( \LS_+^{\ell-1}(P) \right)$ as well.

Now, to prove~\eqref{STgen1eq1}, observe that given a fixed $S \in [p]_q^r$ and for every integer $i \in \set{0, 1, \ldots, q}$,
\[
\sum_{\substack{ T \in [p]_q^r \\ |\sat(T) \cap \sat(S)| = i}} \left( Ye_T \right)[W] = 
\begin{cases}
\binom{qr}{q} \binom{p-qr}{qr-q} |[qr]_q^r| \a_0 &\tn{ if $W=0$;}\\
\a_0 + \binom{qr-i}{q} \binom{p-2qr+i}{qr-q} |[qr]_q^r| \a_0\a_1 &\tn{ if $|\sat(W) \cap \sat(S)| = i$;}\\
 \binom{qr-j}{q} \binom{p-2qr+j}{qr-q} |[qr]_q^r| \a_0\a_1 &\tn{ if $|\sat(r) \cap \sat(i)| = j \neq i$.}
 \end{cases}
\]
Thus, for every $S,W \in [p]_q^r$ where $|\sat(W) \cap \sat(S)| = j$,
\begin{eqnarray*}
&& \sum_{T \in [p]_q^r} \frac{|\sat(S) \cap \sat(T)|}{qr} \left(Ye_T\right)[W] \\
&=& \sum_{i=0}^q \sum_{\substack{ T \in [p]_q^r \\ |\sat(T) \cap \sat(S)| = i}} \frac{i}{qr}\left(Ye_S \right)[W] \\
&=& \frac{j}{qr}\a_0 + \sum_{q=0}^k \frac{q}{qr} \binom{qr-j}{q} \binom{p-2qr+j}{qr-q} |[qr]_q^r| \a_0\a_1\\
&=& \frac{j}{qr}\a_0 + \frac{qr-j}{qr}  \sum_{i=0}^q \binom{qr-j-1}{q-1} \binom{p-2qr+j}{qr-q} |[qr]_q^r| \a_0\a_1\\
&=& \frac{j}{qr}\a_0 + \frac{qr-j}{qr} \binom{p-qr-1}{qr-1} |[qr]_q^r| \a_0\a_1\\
&=& \frac{j}{qr}\a_0 + \frac{qr-j}{qr} \a_0 \\
&=& \a_0 = \left(Ye_0\right)[W].
\end{eqnarray*}
By a similar argument, one can show that
\[
\sum_{T \in [p]_q^r} \frac{|\sat(S) \cap \sat(T)|}{qr} \left( Ye_T\right)[0] = 1 = \left(Ye_0\right)[0],
\]
which completes the proof of~\eqref{STgen1eq1}.

Finally, we show that $Y \succeq 0$. When $p \geq 2qr$, we have
\[
Y = \begin{bmatrix}\frac{qr}{p} \bar{e}^{\top} \\ I \end{bmatrix} Y' \begin{bmatrix} \frac{qr}{p} \bar{e} & I \end{bmatrix}.
\]
Thus, to show that $Y \succeq 0$, it suffices to prove that $Y' \succeq 0$. Observe that
\[
Y' =\a_0 \left( I + \a_1 B_{2qr} \right),
\]
where $B_{2qr} \in \tilde{\M}_{p,q,r}$ was defined before Proposition~\ref{HtildeHoverline}. We also showed in the proof of Proposition~\ref{HtildeHoverline} that $B_{2qr}$ has the same eigenvalues as $J_{p,qr, qr} \otimes J$. From Proposition~\ref{JohnsonEvalues}, $J_{p,qr, qr}$ has eigenvalues $(-1)^j \binom{p-qr-j}{qr-j}$ for $j = 0, \ldots, qr$. Also,  $J$ here is the $|[qr]_q^r|$-by-$|[qr]_q^r|$ matrix of all-ones and thus has eigenvalues $|[qr]_q^r|$ and $0$. Hence, the eigenvalues of $Y'$ are
\[
\a_0 \left( 1+   (-1)^j   \binom{ p-qr-j}{qr-j} |[qr]_q^r| \a_1  \right) = \a_0  \left( 1+   (-1)^j  \frac{\binom{ p-qr-j}{qr-j}}{ \binom{p-qr-1}{qr-1} }  \right),
\]
which are non-negative for all $j \in \set{ 0,1, \ldots ,qr}$. Thus, $Y'$ is positive semidefinite, and so is $Y$. This establishes that $\a_0\bar{e} \in \LS_+^{\ell}(P)$ for all $\ell < \frac{p}{qr}$, and thus shows that $P$ has $\LS_+$-rank at least $\lfloor \frac{p}{qr} \rfloor$.

	We next turn to prove the upper bound on the rank of $P$. By~\cite[Lemma 1.5]{LovaszS91}, if an inequality is valid for $\set{x \in P : x_S =1}$ for all $S \in [p]_q^r$, then it is valid for $\LS_+(P)$. Thus, it follows that $P = \MT_{r}(K_p^q)$ has $\LS_+$-rank at most $1$ if $p < 2qr$. By the same rationale, the lemma implies that if $\MT_{r}(K_{p}^q)$ has $\LS_+$-rank $\ell$, then $\MT_{r}(K_{p+qr}^q)$ has $\LS_+$-rank at most $\ell+1$. Thus, we see that $\MT_{r}(K_p^q)$ has $\LS_+$-rank at most $\lfloor \frac{p}{qr} \rfloor$.
\end{proof}

In general, one of the greatest challenges in establishing lower-bound results for semidefinite lift-and-project relaxations is to verify the positive semidefiniteness of a given family of certificate matrices. In the case of the proof of Theorem~\ref{STgen1}, this task was made relatively simple by observing $Y$ has a full-rank symmetric minor $Y'$ and that $Y'$ is a simple linear combination of associates in $\tilde{\M}_{p,q,r}$.

An immediate implication of the proof of Theorem~\ref{STgen1} is the following integrality gap result on $\MT_{r}(K_p^q)$

\begin{corollary}\label{STgen1cor}
Let $P = \MT_{r}(K_p^q)$ where $p > qr$ and $qr$ does not divide $p$. Then
\[
\frac{\max \set{ \bar{e}^{\top}x :  x \in \LS_+^{\ell}(P)}}{\max \set{ \bar{e}^{\top}x :  x \in P_I}} = \frac{p/qr}{\lfloor p/qr \rfloor} =
1 + \frac{ (p~\tn{mod}~qr)}{ p - (p~\tn{mod}~qr)} 
\]
for all $\ell \in \set{ 0, 1, \ldots, \lfloor \frac{p}{qr} \rfloor -1}$.
\end{corollary}

\begin{proof}
First of all, it is obvious that $\max \set{\bar{e}^{\top}x : x \in P_I} =\lfloor \frac{p}{qr} \rfloor$ and $\max \set{\bar{e}^{\top}x: x \in P} =\frac{p}{qr}$. This establishes the above integrality gap for $\ell=0$. Next, as shown in Theorem~\ref{STgen1}, $\left(\binom{p-1}{qr-1} |[qr]_q^r|\right)^{-1} \bar{e} \in \LS_+^{\ell}(P)$ for all $\ell <  \lfloor \frac{p}{qr} \rfloor $. Thus, the corresponding integrality gap for $\LS_+^{ \lfloor p/qr \rfloor -1}(P)$ is greater than or equal to that of $P$. Since $\LS_+^{\ell+1}(P) \subseteq \LS_+^{\ell}(P)$ for all $\ell$, the integrality gap must be a non-increasing function of $\ell$. This shows that the gap is identical for all values of $\ell  \in \set{0, \ldots, \lfloor \frac{p}{qr} \rfloor -1}$, and our claim follows.
\end{proof}

Next, recall from Proposition~\ref{propnkaJohnson} that the $\LS_+$-rank of the fractional stable set polytope of the graph $\overline{G(J_{p,q,q})}$ is at least $2$ when $p \geq 2q+1$ and is not a multiple of $q$. With Theorem~\ref{STgen1}, we can now determine the exact rank of this set.

\begin{corollary}\label{STgen1cor1b}
Given positive integers $p,q$ where $q \geq 2, p \geq 2q+1$, and $q$ does not divide $p$, the $\LS_+$-rank of $\FRAC(\overline{G(J_{p,q,q})})$ is $\lfloor \frac{p}{q} \rfloor$.
\end{corollary}

\begin{proof}
Notice that $\overline{G(J_{p,q,q})}$ is the line graph of $K_{p}^q$, and so there is a natural one-to-one correspondence between matchings in $K_p^q$ and stable sets in $\overline{G(J_{p,q,q})}$. Therefore, we know that
\[
\MT_1(K_{p}^q)_I = \STAB(\overline{G(J_{p,q,q})}).
\]
Moreover, it is obvious from their definitions that
\[
\MT_1(K_{p}^q) \subseteq \FRAC(\overline{G(J_{p,q,q})}).
\]
Since $\LS_+$ preserves containment, this implies that the $\LS_+$-rank of $\FRAC(\overline{G(J_{p,q,q})})$ is at least that of $\MT_1(K_{p}^q)$. On the other hand, the same rank upper-bound argument in the proof of Theorem~\ref{STgen1} also applies for $\FRAC(\overline{G(J_{p,q,q})})$. Thus, we conclude that when $q \nmid p$, the $\LS_+$-rank of $\FRAC(\overline{G(J_{p,q,q})})$ is exactly $\lfloor \frac{p}{q} \rfloor$.
\end{proof}

Finally, we remark that our lower-bound analysis in the proof of Theorem~\ref{STgen1} also applies to the covering variant of the same problem. If we define the $r$-matching covering polytope to be
\[
\MT^C_{r}(G) := \set{ x \in [0,1]^{E_{r}(G)} : \sum_{\substack{S \in E_{r}(G) \\ S~\tn{saturates}~i}} x_S \geq 1,~\forall i \in V(G)},
\]
then each integral vector in $\MT^C_{r}(G)$ gives a set of $r$-matchings in $G$ whose union form an edge cover. Then we have the following result:

\begin{corollary}\label{STgen1cor2}
Let $P := \MT^C_{r}(K_p^q)$ where $p >2qr$ and $qr$ does not divide $p$. Then the $\LS_+$-rank of $P$ is at least $\lfloor \frac{p}{qr} \rfloor$.
\end{corollary}

\begin{proof}
Following the notation used in the proof of Theorem~\ref{STgen1}, notice that the fractional vector $\bar{x} = \a_0 \bar{e}$ used therein is also contained in $P$, as it satisfies each of the $p$ vertex-incidence constraints of $P$ with equality. Hence, one can use the same certificate matrix $Y$ and induction process to show that $\bar{x} \in \LS_+^{\ell}(P)$, for all $\ell < \lfloor \frac{p}{qr} \rfloor$. When $qr$ does not divide $p$, it is easy to see that 
\[
\max\set{ \bar{e}^{\top}x : x \in P_I} = \left\lceil \frac{p}{qr} \right\rceil > \frac{p}{qr} = \bar{e}^{\top} \bar{x}.
\]
This shows that $P$ has $\LS_+$-rank at least $\lfloor \frac{p}{qr} \rfloor$.
\end{proof}

\subsection{The $b$-hypermatching problem}\label{sec32}

Next, we turn to a different generalization of the classical matching problem, and study its corresponding $\LS_+$-relaxations. Given a $q$-uniform hypergraph $G$ and a positive integer $b$, we say that $S \subseteq E(G)$ is a \emph{$b$-matching} if every vertex has degree at most $b$ in the subgraph of $G$ with edge set $S$. The maximum $b$-matching problem is to find the largest $b$-matching in a given graph. Note that this problem reduces to the maximum matching problem when $b =1$. A natural polyhedral relaxation of this problem is
\[
\bMT(G) := \set{ x \in [0,1]^{E(G)} : \sum_{S \in E(G), S \ni i} x_S \leq b,~\forall i \in V(G)}.
\]
Then there is a one-to-one correspondence between the integral vectors in $\bMT(G)$ and the $b$-matchings of $G$. 

For the complete graph $G= K_p^q$ and any integer $b < p$, it is easy to see that there exists a $b$-regular subgraph in $G$ if and only if $bp$ is a multiple of $q$, in which case an optimal $b$-matching would contain exactly $\frac{bp}{q}$ hyperedges. On the other hand, if $q \nmid bp$, then the largest $b$-matching has size $\lfloor \frac{bp}{q}\rfloor$. Now since $b \binom{p-1}{q-1}^{-1} \bar{e} \in \bMT(G)$,  $\max\set{ \bar{e}^{\top}x : x \in \bMT(G)} \geq \frac{bp}{q}$, and so $\bMT(G) \neq (\bMT(G))_I$ when $q$ does not divide $bp$. 

Recently, Kurpisz et al.~\cite{KurpiszLM18} also studied lift-and-project relaxations of the $b$-matching problem on hypergraphs. Therein, their focus is on the $\Las$ operator (due to Lasserre~\cite{Lasserre01} and yields tighter relaxations than $\LS_+$ in general),  applied to a relaxation that is $\bMT(K_p^q)$ with an additional linear constraint. They proved a $\Las$-rank upper bound of $\max\set{b, \frac{1}{2} \lfloor \frac{bp}{q} \rfloor}$ for that relaxation, and that the bound is tight in the cases where $b=1, q=2$ and $4 | p-1$. Their results are incomparable with Theorem~\ref{STgen4} below.

Herein, we focus on the $\LS_+$-rank of $\bMT(K_p^q)$. First, we remark that the ideas in establishing the upper bound in the proof of Theorem~\ref{STgen1} can be used to show that $\bMT(K_p^q)$ has $\LS_+$-rank at most $\lfloor \frac{bp}{q} \rfloor$. Next, we establish a lower bound for the $\LS_+$-rank below, showing that this is another case where $\LS_+$ is not efficient at computing the integer hull of the given relaxation. Once again, a key component of our proof involves the known eigenvalues of the Johnson scheme.

\begin{theorem}\label{STgen4}
Let $b,p,q$ be positive integers where $q$ does not divide $bp$ and $q^2 \geq b$. Then the $\LS_+$-rank of $\bMT(K_p^q)$ is at least $\lfloor \frac{p-b-q+1}{2q}\rfloor +1$ for all $p \geq b+q-1$.
\end{theorem}

\begin{proof}
For convenience, let $P := \bMT(K_p^q), \a_0 := b \binom{p-1}{q-1}^{-1}$, and $\ell := \lfloor \frac{p-b-q+1}{2q} \rfloor$. By induction on $\ell$, we shall show that $\a_0 \bar{e} \in \LS_+^{\ell}(P)$. Since $\a_0 \bar{e} \not\in P_I$ when $q$ does not divide $bp$, the claim above would imply that $P$ has $\LS_+$-rank at least $\ell+1$.

For the base case $\ell = 0$, since $p \geq b+q-1$, $b \binom{p-1}{q-1}^{-1} \leq 1$, and so $b \binom{p-1}{q-1}^{-1} \bar{e} \in P = \LS_+^0(P)$.

Next, for the inductive step, we aim to show that $\a_0 \bar{e} \in \LS_+^{\ell} (P)$ follows from the inductive hypothesis $b\binom{p-2q-1}{q-1}^{-1} \bar{e} \in \LS_+^{\ell-1} \left(\bMT(K_{p-2q}^q)\right)$. Define $\a_1 := \frac{b-1}{q(p-q)}, \a_2:= \frac{ b-1}{q \binom{p-q}{q-1}}$, and $\a_3:= \frac{bp-2bq+q}{q \binom{p-q}{q}}$, and the certificate matrix  
\[
Y = \begin{bmatrix}
1 &\a_0 \bar{e}^{\top} \\
\a_0\bar{e}& Y'
\end{bmatrix}
\]
where $Y'$ is the $|[p]_q|$-by-$|[p]_q|$ matrix with entries 
\[
Y'[i,j] = 
\begin{cases}
\a_0  & \tn{if $i=j$;}\\
\a_0\a_1 & \tn{if $|i \cap j| = q-1$;}\\
\a_0\a_2 & \tn{if $|i \cap j| = 1$;}\\
\a_0\a_3 & \tn{if $i \cap j = \es$;}\\
0 & \tn{otherwise.}
\end{cases}
\]
We show that $Y$ satisfies all conditions imposed by $\LS_+$. First, it is apparent that $Y=Y^{\top}$ and $Ye_0 =  \tn{diag}(Y) = \begin{bmatrix} 1 \\ \a_0 \bar{e} \end{bmatrix}$. Next, we show that $Ye_i \in K\left( \LS_+^{\ell-1}(P) \right)$ for every edge $i$. Given two disjoint edges $i,j$, define
\[
\S_{i,j} := \set{ \set{v_1, v_2} : v_1 \in i, v_2 \in j}.
\]
That is, $\S_{i,j}$ consists of the $q^2$ 2-vertex sets where one vertex belongs to $i$ and the other belongs to $j$. Then, given $i,j$ and $S \subseteq \S_{i,j}$ where $|S| = b-1$, we construct a vector $w_{i,j,S} \in \mR^{[p]_q}$ as follows:
\[
w_{i,j,S}[h] = \begin{cases}
1 & \tn{if $h =i$ or $h=j$;}\\
1 & \tn{if $h= (i \sm s) \cup (s \cap j)$ for some $s \in S$;}\\
1 & \tn{if $h= (j \sm s) \cup (s \cap i)$ for some $s \in S$;}\\
b \binom{p-2q-1}{q-1}^{-1} & \tn{if $h$ is disjoint from $i \cup j$;}\\
0 & \tn{otherwise.}\\
\end{cases}
\]
Notice that the hyperedges receiving $1$'s in $w_{i,j,S}$ form a $b$-regular subgraph on the vertices $i \cup j$. Also, choosing $S$ requires $b-1 \leq q^2$, which is implied by the assumption $b \leq q^2$. 

Then we know by the inductive hypothesis that $w_{i,j,S} \in \LS_+^{\ell-1}(P)$. Next, consider the vector
\[
z := \frac{1}{\binom{p-q}{q}\binom{q^2}{b-1}} \sum_{\substack{j \in E(G) \\ j \cap i = \es}} \sum_{\substack{S \subseteq \S_{i,j}\\ |S| = b-1}} w_{i,j,S}.
\]
Since $w_{i,j,S}[i] = 1$ for all $j,S$ included in the sum above, $z[i] = 1$. Likewise, for all $h \in E(G)$ where $2 \leq |h \cap i| \leq q-2, w_{i,j,S}[h] = 0$. Hence, $z[h] =0$ for these edges as well. Next, by symmetry of the underlying complete graph $G$, we know that $z[h] = z[h']$ if there is an automorphism on $G$ that maps vertices in $i$ to itself while mapping $h$ to $h'$. Thus, there must exist constants $\b_1, \b_2, \b_3$ such that 
\[
z[h] =
\begin{cases}
 \b_1 & \tn{if $|h \cap i| = q-1$;}\\
 \b_2& \tn{if $|h \cap i| = 1$;}\\
 \b_3 & \tn{if $|h \cap i| = 0$.}
\end{cases}
\]
 Now, notice that $\sum_{h \in E(G), |h \cap i| =q-1} w_{i,j,S} = b-1$ for all $j$ and $S$. Thus,
\[
\b_1 = \frac{b-1}{ |\set{h \in E(G) : |h \cap i| = q-1}|} = \frac{b-1}{q(p-q)} = \a_1.
\]
One can likewise show that $\b_2=\a_2$ and $\b_3 = \a_3$, which shows that $Ye_i = \a_0 \begin{bmatrix} 1 \\ z \end{bmatrix}$. Since $z \in \LS_+^{\ell-1}(P)$ (due to it being a convex combination of $w_{i,j,S}$'s, points inside the convex set $\LS_+^{\ell-1}(P)$), we obtain that $Ye_i \in K \left( \LS_+^{\ell-1}(P)\right)$.

We next show that $Y(e_0-e_i) \in K\left( \LS_+^{\ell-1}(P) \right)$ with a similar argument. Given disjoint edges $i,j$ and $S \subseteq \S_{i,j}$ where $|S| = b$, define $\overline{w}_{i,j,S} \in \mR^{[p]_q}$ such that
\[
\overline{w}_{i,j,S}[h] = \begin{cases}
1 & \tn{if $h= (i \sm s) \cup (s \cap j)$ for some $s \in S$;}\\
1 & \tn{if $h= (j \sm s) \cup (s \cap i)$ for some $s \in S$;}\\
b \binom{p-2q-1}{q-1}^{-1} & \tn{if $h$ is disjoint from $i \cup j$;}\\
0 & \tn{otherwise.}\\
\end{cases}
\]
(This construction is where we need $q^2 \geq b$.) Next, if we define 
\[
\overline{z} := \frac{1}{\binom{p-q}{q}\binom{q^2}{b}} \sum_{\substack{j \in E(G) \\ j \cap i = \es}} \sum_{\substack{S \subseteq \S_{i,j}\\ |S| = b}} \overline{w}_{i,j,S},
\]
one can apply a similar argument to the above and deduce that 
\[
\overline{z}[h]= \begin{cases}
0 & \tn{if $h=i$;}\\
0 & \tn{if $2\leq |h \cap i| \leq q-2$;}\\
\frac{b}{q(p-q)} =  \frac{\a_0}{1-\a_0}(1-\a_1) & \tn{if $|h \cap i|= q-1$;}\\
\frac{b}{q\binom{p-q}{q-1}} =  \frac{\a_0}{1-\a_0}(1-\a_2) & \tn{if $|h \cap i|= 1$;}\\
\frac{bp-2bq}{q \binom{p-q}{q}} =  \frac{\a_0}{1-\a_0}(1-\a_3) & \tn{if $|h \cap i| = 0$.}
\end{cases}
\]
Thus, $Y(e_0 - e_i) = (1-\a_0) \begin{bmatrix} 1 \\ \overline{z} \end{bmatrix} \in K\left( \LS_+^{\ell-1}(P) \right)$.

Finally, we show that $Y \succeq 0$. Observe that
\[
Y = \begin{bmatrix} \frac{q}{bp} \bar{e}^{\top} \\ I \end{bmatrix} Y' \begin{bmatrix} \frac{q}{bp} \bar{e} & I \end{bmatrix}.
\]
Thus, to show that $Y \succeq 0$, it suffices to prove that $Y' \succeq 0$. Now, notice that $Y' \in \tn{Span}~\J_{p,q}$. In fact,
\[
Y' =\a_0 \left( I + \a_1 J_{p,q,1}+ \a_2 J_{p,q,q-1} + \a_3 J_{p,q,q} \right).
\]
Applying Proposition~\ref{JohnsonEvalues} (using formula~\eqref{JohnsonEvalues1} for the case $i=q-1$ and~\eqref{JohnsonEvalues2} for the case $i=1$), one obtains that the eigenvalues of $Y'$ are
\begin{eqnarray}
\nonumber && \a_0 \bigg(1 + \a_1((q-j)(p-q-j) - j)  \\
\nonumber && +\a_2 \left((-1)^j(p-2q+2))\binom{p-q+1-j}{q-1-j} + (-1)^{j+1} q \binom{p-q-j}{q-j} \right) \\
\nonumber && + \a_3 (-1)^j \binom{p-q-j}{q-j}\bigg)  \\
\label{STgen4evalue}&=& \a_0 \left( 1+ (-1)^j\frac{\binom{p-q-j}{q-j}}{\binom{p-q-1}{q-1}} \right) \left( 1+ \frac{(b-1)((q-j)(p-q-j)-j)}{q(p-q)} \right).
\end{eqnarray}
Notice that~\eqref{STgen4evalue} is non-negative for all $p \geq 2q$ and for all $j < q$. When $j=q$, ~\eqref{STgen4evalue} is non-negative when 
\[
\left( 1+ \frac{(b-1)((q-q)(p-q-q)-q)}{q(p-q)} \right) \geq 0 \iff p \geq q+b-1,
\]
which is an assumption in the hypothesis. This finishes our proof.
\end{proof}

We remark that the lower bound given in Theorem~\ref{STgen4} is not always tight. For instance, when $b=1$, the theorem gives $\bMT(K_p^q)$ has $\LS_+$-rank at least $\lfloor \frac{p-q}{2q} \rfloor + 1$, while Theorem~\ref{STgen1} (specialized to $r=1$) gives a better rank lower bound of $\lfloor \frac{p}{q} \rfloor$. It is possible one can improve the bound in Theorem~\ref{STgen4} (and/or weaken the assumption $b \leq q^2$) by using a different certificate matrix, perhaps by involving more associates in the Johnson scheme. Of course, this could potentially lead to a more challenging analysis of its eigenvalues.

Also, for the same reason why Theorem~\ref{STgen1} implies Corollary~\ref{STgen1cor2}, the proof of Theorem~\ref{STgen4} can be easily adapted to show that the same $\LS_+$-rank lower bound applies for the covering variant of the $b$-matching problem.

\begin{corollary}\label{STgen4cor1}
Let $b,p,q$ be positive integers where $q$ does not divide $bp$ and $q^2 \geq b$, and let $G = K_p^q$. Then the $\LS_+$-rank of 
\[
\bMT^C(G) := \set{ x \in [0,1]^{E(G)} : \sum_{S \in E(G), S \ni i} x_S \geq b,~\forall i \in V(G)}
\]
is at least $\lfloor \frac{p-b-q+1}{2q}\rfloor +1$ for all $p \geq b+q-1$.
\end{corollary}

\ignore{
Would be nice if I could find an example where $\binom{n-1}{k-1}^{-1}\bar{e} \in \LS_+^p(\MT(K^k_n))$ but $b\binom{n-1}{k-1}^{-1}\bar{e} \in \LS_+^p(\bMT(K^k_n))$, or vice versa.
}

\section{More on the hypermatching homogeneous coherent configuration $\M_{p,q,r}$}\label{secHnkl}

After working with the simple commutative subscheme $\tilde{\M}_{p,q,r}$ in Section~\ref{secLandP}, we now look into the full h.c.c.\ $\M_{p,q,r}$ and try to gain a better understanding of it. In Section~\ref{sec41}, we discuss some combinatorial characterizations of the associates of $\M_{p,q,r}$, and in particular enumerate the associates in $\M_{p,2,r}$ via counting a certain type of integer partitions. This will in turn help us study which contractions of $\M_{p,2,r}$ result in symmetric subschemes in Section~\ref{sec42}.

\subsection{Characterizing associates in $\M_{p,q,r}$}\label{sec41}

Herein, we look into characterizing associates in $\M_{p,q,r}$ via some familiar combinatorial objects. For convenience, let $a_{p,q,r}$ be the number of equivalence classes in the relation defined in Definition~\ref{defEquivRel}. We first focus on ordinary graphs (i.e., the case $q=2$) and map associates in $\M_{p,2,r}$ to a specific type of integer partitions, before returning to discuss the case for arbitrary $q$ later in this subsection.

When $q=2$ and $p=2r$, it is known that there is a one-to-one correspondence between the isomorphism classes and even partitions of $2r$ (i.e., the number ways to write $2r$ as the sum of a non-increasing sequence of even positive integers). The following result extends this to general values of $p$.

\begin{proposition}\label{propEquivPartitions}
For all positive integers $p,r$ where $p \geq 2r$, $a_{p,2,r}$ is equal to the number of partitions of $2r$ with four types of parts $\set{\l^+, \l^-,\overline{\l}, \l' \geq 1}$, such that
\begin{itemize}
\item[(P1)]
the parts of the types $\l^+,\l^-$ are all odd, and the parts of the types $\overline{\l}, \l'$ are all even;
\item[(P2)]
the number of parts of type $\l^+$ is equal to that of type $\l^-$; and
\item[(P3)]
there are at most $p-2r$ total number of parts from the types $\l^+, \l^-, \overline{\l}$ combined.
\end{itemize}
\end{proposition}

\begin{proof}
We construct a bijection between the equivalence classes and the set of partitions described in our claim.
For each equivalence class $X_i$, take any element $(S,T) \in X_i$, and consider the components in the subgraph formed by the edges in $S \cup T$. Notice that since $S,T$ is each a matching, every vertex in this subgraph has degree at most $2$. Now, for each component that contains $\l$ edges, we assign it to a part as follows:
\begin{itemize}
\item
$\l^+$ if the component is a path of odd length, with $\frac{\l+1}{2}$ edges coming from $S$.
\item
$\l^-$ if the component is a path of odd length, with $\frac{\l-1}{2}$ edges coming from $S$.
\item
$\overline{\l}$ if the component is a path of even length.
\item
$\l'$ if the component is a cycle of length $2\l$. This includes the case of $2$-cycles, which occurs when the component consists of two overlapping edges, one from $S$ and one from $T$.
\end{itemize}

Figure~\ref{fig1} illustrates the correspondence between the partitions and equivalence classes for the case $r=2$.

If we do that for each component in $S \cup T$, we obtain parts that add up to $2r$ (since the value of each part is equal to the number of edges in the corresponding component), and the partition satisfies (P1) by construction. Next, (P2) holds since $|S| = |T|= r$ and the number of odd paths with one more edge from $S$ must be equal to the number of odd paths with one more edge from $T$. Also, notice that a component corresponding to parts $\l^+, \l^-, \overline{\l}$ saturates $\l+1$ vertices, while a component corresponding to $\l'$ has exactly $\l$ vertices. Thus, the total number of vertices saturated by all components is $2r$ plus the number of non-primed parts in the partition. Therefore, these components do not occupy more than $p$ vertices if and only if the number of the non-primed parts is no more than $p-2r$, satisfying (P3).

Note that the construction of the partition is reversible --- given a partition of $2r$ with the aforementioned four kinds of parts and the given conditions, we can uniquely recover the types of components in the graph with the edges $S \cup T$, and thus the equivalence class $X_i$. This finishes our proof.
\end{proof}

\begin{figure}[htb]
\begin{center}
\footnotesize
\begin{tabular}{cccccc}
\quad
\begin{tikzpicture}[scale =1 ,>=stealth',shorten >=1pt,auto,node distance=4cm,
  thick,main node/.style={circle, scale = 0.3, draw}]

  \node[main node] at (0,0) (01) {};
  \node[main node] at (1,0) (11) {};
  \node[main node] at (0,1) (02) {};
  \node[main node] at (1,1) (12) {};
  \node[main node] at (0,2) (03) {};
  \node[main node] at (1,2) (13) {};
  \node[main node] at (0,3) (04) {};
  \node[main node] at (1,3) (14) {};

 \path
(01) edge[very thin] (11)
(02) edge[very thin] (12)
(02) edge[dotted, ultra thick] (12)
(01) edge[dotted, ultra thick] (11);
\end{tikzpicture}
\quad
&
\quad
\begin{tikzpicture}[scale =1 ,>=stealth',shorten >=1pt,auto,node distance=4cm,
  thick,main node/.style={circle, scale = 0.3, draw}]

  \node[main node] at (0,0) (01) {};
  \node[main node] at (1,0) (11) {};
  \node[main node] at (0,1) (02) {};
  \node[main node] at (1,1) (12) {};
  \node[main node] at (0,2) (03) {};
  \node[main node] at (1,2) (13) {};
  \node[main node] at (0,3) (04) {};
  \node[main node] at (1,3) (14) {};

 \path
(01) edge[very thin] (11)
(02) edge[very thin] (12)
(02) edge[dotted, ultra thick] (01)
(12) edge[dotted, ultra thick] (11);
\end{tikzpicture}
\quad
&
\quad
\begin{tikzpicture}[scale =1 ,>=stealth',shorten >=1pt,auto,node distance=4cm,
  thick,main node/.style={circle, scale = 0.3, draw}]

  \node[main node] at (0,0) (01) {};
  \node[main node] at (1,0) (11) {};
  \node[main node] at (0,1) (02) {};
  \node[main node] at (1,1) (12) {};
  \node[main node] at (0,2) (03) {};
  \node[main node] at (1,2) (13) {};
  \node[main node] at (0,3) (04) {};
  \node[main node] at (1,3) (14) {};

 \path
(01) edge[very thin] (11)
(02) edge[very thin] (12)
(01) edge[dotted, ultra thick] (11)
(02) edge[dotted, ultra thick] (03);
\end{tikzpicture}
\quad
&
\begin{tikzpicture}[scale =1 ,>=stealth',shorten >=1pt,auto,node distance=4cm,
  thick,main node/.style={circle, scale = 0.3, draw}]

  \node[main node] at (0,0) (01) {};
  \node[main node] at (1,0) (11) {};
  \node[main node] at (0,1) (02) {};
  \node[main node] at (1,1) (12) {};
  \node[main node] at (0,2) (03) {};
  \node[main node] at (1,2) (13) {};
  \node[main node] at (0,3) (04) {};
  \node[main node] at (1,3) (14) {};

 \path
(01) edge[very thin] (11)
(02) edge[very thin] (12)
(02) edge[dotted, ultra thick] (03)
(12) edge[dotted, ultra thick] (11);
\end{tikzpicture}
\quad
&
\quad
\begin{tikzpicture}[scale =1 ,>=stealth',shorten >=1pt,auto,node distance=4cm,
  thick,main node/.style={circle, scale = 0.3, draw}]

  \node[main node] at (0,0) (01) {};
  \node[main node] at (1,0) (11) {};
  \node[main node] at (0,1) (02) {};
  \node[main node] at (1,1) (12) {};
  \node[main node] at (0,2) (03) {};
  \node[main node] at (1,2) (13) {};
  \node[main node] at (0,3) (04) {};
  \node[main node] at (1,3) (14) {};

 \path
(01) edge[very thin] (11)
(02) edge[very thin] (12)
(03) edge[dotted, ultra thick] (13)
(01) edge[dotted, ultra thick] (11);
\end{tikzpicture}
\quad
&
\quad
\begin{tikzpicture}[scale =1 ,>=stealth',shorten >=1pt,auto,node distance=4cm,
  thick,main node/.style={circle, scale = 0.3, draw}]

  \node[main node] at (0,0) (01) {};
  \node[main node] at (1,0) (11) {};
  \node[main node] at (0,1) (02) {};
  \node[main node] at (1,1) (12) {};
  \node[main node] at (0,2) (03) {};
  \node[main node] at (1,2) (13) {};
  \node[main node] at (0,3) (04) {};
  \node[main node] at (1,3) (14) {};

 \path
(01) edge[very thin] (11)
(02) edge[very thin] (12)
(01) edge[dotted, ultra thick, bend left] (03)
(12) edge[dotted, ultra thick] (13);
\end{tikzpicture}
\quad
\\
\vspace{0.5cm}
 $X_0 : (2', 2')$   &   $X_1 : (4')$   &   $X_2 : (\overline{2},2')$   &   $X_3 :  (\overline{4})$   &   $ X_4 :  (1^+, 1^-, 2')$   &   $X_5 :  (\overline{2}, \overline{2})$    \\

\begin{tikzpicture}[scale =1 ,>=stealth',shorten >=1pt,auto,node distance=4cm,
  thick,main node/.style={circle, scale = 0.3, draw}]

  \node[main node] at (0,0) (01) {};
  \node[main node] at (1,0) (11) {};
  \node[main node] at (0,1) (02) {};
  \node[main node] at (1,1) (12) {};
  \node[main node] at (0,2) (03) {};
  \node[main node] at (1,2) (13) {};
  \node[main node] at (0,3) (04) {};
  \node[main node] at (1,3) (14) {};

 \path
(01) edge[very thin] (11)
(02) edge[very thin] (12)
(03) edge[dotted, ultra thick] (13)
(01) edge[dotted, ultra thick] (02);
\end{tikzpicture}
\quad
&
\quad
\begin{tikzpicture}[scale =1 ,>=stealth',shorten >=1pt,auto,node distance=4cm,
  thick,main node/.style={circle, scale = 0.3, draw}]

  \node[main node] at (0,0) (01) {};
  \node[main node] at (1,0) (11) {};
  \node[main node] at (0,1) (02) {};
  \node[main node] at (1,1) (12) {};
  \node[main node] at (0,2) (03) {};
  \node[main node] at (1,2) (13) {};
  \node[main node] at (0,3) (04) {};
  \node[main node] at (1,3) (14) {};

 \path
(01) edge[very thin] (11)
(02) edge[very thin] (12)
(02) edge[dotted, ultra thick] (03)
(12) edge[dotted, ultra thick] (13);
\end{tikzpicture}
\quad
&
\quad
\begin{tikzpicture}[scale =1 ,>=stealth',shorten >=1pt,auto,node distance=4cm,
  thick,main node/.style={circle, scale = 0.3, draw}]

  \node[main node] at (0,0) (01) {};
  \node[main node] at (1,0) (11) {};
  \node[main node] at (0,1) (02) {};
  \node[main node] at (1,1) (12) {};
  \node[main node] at (0,2) (03) {};
  \node[main node] at (1,2) (13) {};
  \node[main node] at (0,3) (04) {};
  \node[main node] at (1,3) (14) {};

 \path
(01) edge[very thin] (11)
(02) edge[very thin] (12)
(13) edge[dotted, ultra thick] (04)
(02) edge[dotted, ultra thick] (03);
\end{tikzpicture}
\quad

&
\quad
\begin{tikzpicture}[scale =1 ,>=stealth',shorten >=1pt,auto,node distance=4cm,
  thick,main node/.style={circle, scale = 0.3, draw}]

  \node[main node] at (0,0) (01) {};
  \node[main node] at (1,0) (11) {};
  \node[main node] at (0,1) (02) {};
  \node[main node] at (1,1) (12) {};
  \node[main node] at (0,2) (03) {};
  \node[main node] at (1,2) (13) {};
  \node[main node] at (0,3) (04) {};
  \node[main node] at (1,3) (14) {};

 \path
(01) edge[very thin] (11)
(02) edge[very thin] (12)
(03) edge[dotted, ultra thick] (13)
(04) edge[dotted, ultra thick] (14);
\end{tikzpicture}
\quad
&
&

\begin{tikzpicture}[scale =1 ,>=stealth',shorten >=1pt,auto,node distance=4cm,
  thick,main node/.style={circle, scale = 0.3, draw}]
\def\x{0.2}
  \node at (0 - \x,1) (01) {};
  \node at (1+\x,1) (11) {};
  \node at (0- \x,2.5) (02) {};
  \node at (1+ \x,2.5) (12) {};

\node at (0.5,0.5) (a) {{\footnotesize edges in $S$}};
\node at (0.5,2) (a) {{\footnotesize edges in $T$}};

 \path
(01) edge[very thin] (11)
(02) edge[dotted, ultra thick] (12);
\end{tikzpicture}

\\
\vspace{0.5cm} 
  $X_6 : (3^+,1^-)$   &    $X_7 : (1^+, 3^-)$   &    $X_8 : (1^+,1^-,\overline{2})$   &    $X_9 : (1^+,1^+,1^-,1^-)$    &&
\end{tabular}
\caption{The bijection between integer partitions and non-isomorphic unions of two matchings in $[p]_{2}^2$.}\label{fig1}
\end{center}
\end{figure}

Observe that, when $p=2r$, (P3) assures that the corresponding partitions all have only primed parts, and so $a_{2r,2,r}$ is indeed the number of even partitions of $2r$. As $p$ increases from $2r$ to $4r$, so does $a_{p,2,r}$. However, notice that $a_{p,2,r}$ is constant for all $p \geq 4r$. In fact, it follows from Proposition~\ref{propEquivPartitions} that
\begin{equation}\label{anlcount}
a_{p,2,r} = [x^{2r}y^{r}] \prod_{i \geq 1} \left( \frac{1}{(1-x^{2i-1}y^i)(1-x^{2i-1}y^{i-1})(1-x^{2i}y^i)^2} \right)
\end{equation}
for all $p \geq 4r$. Here, the degree of $x$ counts the total number of edges in a component, the degree of $y$ counts the number of edges in the component that belong to $S$, and the generating functions $\frac{1}{1-x^{2i-1}y^i}, \frac{1}{1-x^{2i-1}y^{i-1}}, \frac{1}{1-x^{2i}y^i}$, and $\frac{1}{1-x^{2i}y^i}$ correspond to parts of types $\l^+, \l^-, \overline{\l}$, and $\l'$ respectively. Using~\eqref{anlcount}, we determine that the first few terms of the sequence $\set{a_{4r,2,r}}_{r \geq 0}$ are
\[
\begin{array}{r|c|c|c|c|c|c|c|c|c|c|c|c|c|c}
r & 0 & 1 & 2 & 3 & 4 & 5 & 6 & 7 & 8 & 9 & 10 & 11 & 12 & 13\\
\hline
a_{4r,2,r} & 1 & 3 & 10 & 27 & 69 & 161 & 361 & 767 & 1578 & 3134 & 6064 & 11432 & 21105 & 38175  
\end{array}
\]
This sequence was previously unreported to the Online Encyclopedia of Integer Sequences (OEIS)~\cite{OEIS}, now sequence A316587 therein.

We next describe an alternative approach to characterizing the associates of the hypermatching h.c.c.\ using equivalence classes of meet tables. While this approach boosts the advantage of applying for all values of $q$ (unlike the integer partitions approach above that is restricted to the case $q=2$), enumerating the equivalence classes of meet tables is seemingly difficult.

Given matchings $S, S' \in [p]_q^r$ where $S = \set{S_1, \ldots, S_r}$ and $S' = \set{S'_1, \ldots, S'_r}$, define the \emph{meet table} of $S$ and $S'$ to be the $r$-by-$r$ matrix $M_{S,S'}$ where
\[
M_{S,S'}[i,j] = | S_i \cap S'_j|
\]
for all $i, j \in [r]$. Notice that the matrix $M_{S,S'}$ must satisfy the following properties:
\begin{itemize}
\item[(T1)]
Every entry of the matrix is an integer between $0$ and $q$.
\item[(T2)]
The entries in every row and every column sum to no more than $q$.
\item[(T3)]
The $r^2$ entries of the matrix sum to at least $2qr-p$.
\end{itemize}

Conversely, given an $r$-by-$r$ matrix $T$ that satisfies properties (T1)-(T3), one can find $S, S' \in [p]_q^r$ such that $T = M_{S,S'}$. Next, we say that two meet tables $T,T'$ are related if there exist $r$-by-$r$ permutation matrices $P, P'$ such that $T' = PTP'^{\top}$. In other words, $T,T'$ are related if one matrix can be obtained from the other by permuting rows and columns. Then it is not hard to see that given matchings $S_1, S_1', S_2, S_2' \in [p]_q^r$, $(S_1, S_1')$ and $(S_2, S_2')$ belong to the same equivalence class (as defined in Definition~\ref{defEquivRel}) if and only if $M_{S_1, S_1'} \sim M_{S_2, S_2'}$. For instance, the 10 isomorphism classes of $\M_{p, 2,2}$ illustrated in Figure~\ref{fig1} correspond to the following meet tables:
\[
\begin{array}{lllll}
X_0 : \begin{bmatrix}
2 & 0 \\ 0 & 2
\end{bmatrix},
&
X_1 : \begin{bmatrix}
1 & 1 \\ 1 & 1
\end{bmatrix},
&
X_2 : \begin{bmatrix}
2 & 0 \\ 0 & 1
\end{bmatrix},
&
X_3 : \begin{bmatrix}
1 & 1 \\ 0 & 1
\end{bmatrix},
&
X_4 : \begin{bmatrix}
2 & 0 \\ 0 & 0
\end{bmatrix},
\\
X_5 : \begin{bmatrix}
1 & 0 \\ 0 & 1
\end{bmatrix},
&
X_6 : \begin{bmatrix}
1 & 1 \\ 0 & 0
\end{bmatrix},
&
X_7 : \begin{bmatrix}
0 & 1 \\ 0 & 1
\end{bmatrix},
&
X_8 : \begin{bmatrix}
0 & 1 \\ 0 & 0
\end{bmatrix},
&
X_9 : \begin{bmatrix}
0 & 0 \\ 0 & 0
\end{bmatrix}.
\end{array}
\]
Thus, the problem of counting the number of associates in $\M_{p,q,r}$ can be solved by enumerating the equivalence classes of meet tables that satisfy (T1)-(T3).

We do so for the special case of $r=2$. Recall that $a_{p,q,r}$ denotes the number of equivalence classes in $\M_{p,q,r}$. Then we have the following:

\begin{proposition}\label{propap2rcount}
Given positive integers $p,q$ where $p \geq 4q$,
\[
a_{p,q,2} = 
\begin{cases}
\frac{1}{24}(q^4+6q^3+20q^2+36q+24) & \tn{if $q$ is even;}\\
\frac{1}{24}(q^4+6q^3+20q^2+30q+15) & \tn{if $q$ is odd.}\\
\end{cases}
\]
\end{proposition}

\begin{proof}
We count the number of equivalence classes of $2$-by-$2$ meet tables $M$ that satisfy properties (T1)-(T3) by cases. First, if $M$ has at least one zero entry, then it is related to one of the following:
\begin{itemize}
\item
$\begin{bmatrix} a & 0 \\ 0 &b \end{bmatrix}$ where $q \geq a \geq b \geq 0$. This gives $\binom{q+2}{2}$ possibilities.
\item
$\begin{bmatrix} a & b \\ 0 &0 \end{bmatrix}$ or $\begin{bmatrix} a & 0 \\ b &0 \end{bmatrix}$ where $a \geq b \geq 1$ and $a+b \leq q$. This gives $2 \left( \frac{q^2}{4} \right)$ possibilities when $q$ is even, and $2 \left( \frac{q^2-1}{4}\right)$ possibilities when $q$ is odd.
\item
$\begin{bmatrix} a & b \\ c &0 \end{bmatrix}$ where $a,b,c \geq 1, a+b \leq q$, and $a+c\leq q$. For each fixed $a \in [q-1]$ there are $q-a$ choices for each of $b$ and $c$. Thus, this gives
\[
\sum_{a=1}^{q-1} (q-a)^2 = \frac{q(q-1)(2q-1)}{6}
\]
possibilities.
\end{itemize}
Now suppose $M = \begin{bmatrix} a & b \\ c &d \end{bmatrix}$ where $a,b,c,d \geq 1$. Note that $M' := \begin{bmatrix} a-1 & b-1 \\ c-1 &d-1 \end{bmatrix}$ would be a meet table for some equivalence class in $\M_{p-4,q-2,2}$. Moreover, the correspondence is bijective. So this gives $a_{p-4,q-2,2}$ possibilities.

Thus, when $q$ is even, the total number of equivalence classes is
\begin{eqnarray*}
a_{p,q,2} &=& \binom{q+2}{2} + 2 \left( \frac{q^2}{4} \right) + \frac{q(q-1)(2q-1)}{6} + a_{p-4,q-2,2}\\
&=& \frac{1}{6} (2q^3+3q^2+10q+6) + a_{p-4,q-2,2}.
\end{eqnarray*}
Likewise, when $q$ is odd, we obtain $a_{p,q,2} = \frac{1}{6} (2q^3+3q^2+10q+3) + a_{p-4,q-2,2}$. Using $a_{p,0,2} = 1$ for all $p \geq 0$ and $a_{p,1,2} = 3$ for all $p \geq 4$, we obtain the formulas as claimed by solving a simple recurrence for each parity of $q$.
\end{proof}

Using Proposition~\ref{propap2rcount}, we obtain that the first few terms of the sequence $\set{a_{4q,q,2}}_{q \geq 0}$ are
\[
\begin{array}{r|c|c|c|c|c|c|c|c|c|c|c|c|c|c|c}
q & 0 & 1 & 2 & 3 & 4 & 5 & 6 & 7 & 8 & 9 & 10 & 11 & 12 & 13\\
\hline
a_{4q,q,2} &
 1 & 3 & 10 & 22 & 47 &
  85 & 148 & 236 & 365 & 535 & 766 & 1058 &  1435 & 1897\\
\end{array}
  \]
This sequence was also not previously reported to the OEIS, now sequence A336529 therein.

While the proof of Proposition~\ref{propap2rcount} is elementary, it is also rather ad hoc and does not seem easily extendable to obtain a formula for $a_{p,q,r}$ for general $r$, which may require a more sophisticated approach.

\subsection{Symmetric subschemes of $\M_{p,q,r}$}\label{sec42}

As seen in the analyses of lift-and-project relaxations in Section~\ref{secLandP}, it is much easier to work with a commutative scheme where the eigenspaces of the associates are aligned. Moreover, many lift-and-project operators (including $\LS_+$) require its certificate matrices to be symmetric. This naturally raises the question of when $\M_{p,q,r}$ is indeed a symmetric h.c.c.\ (which would imply that it is also a commutative scheme), and also, which contractions of associates in $\M_{p,q,r}$ would lead to symmetric subschemes.

We have already seen that when $r=1$, $\M_{p,q,r}$ reduces to the Johnson scheme $\J_{p,q}$, which is obviously symmetric and commutative. For $q=2$ and arbitrary $r$, it is known that $\M_{p,2,r}$ is a commutative scheme if and only if $p \in \set{2r, 2r+1}$~\cite{GodsilM10}. We provide an elementary proof of this below:

\begin{proposition}\label{propNoncomm}
Suppose $r \geq 2$ is a fixed integer. Then $\M_{p,2,r}$ is a commutative scheme if and only if $p \in \set{2r, 2r+1}$.
\end{proposition}

\begin{proof}
First, suppose $p \in \set{2r, 2r+1}$. In this case, given $(S,T)$ in any equivalence class $X_i$, the components in $S \cup T$ consist of at most one path (which must be even), with the rest all being even cycles. Then we see that $(S,T)$ and $(T,S)$ belong to the same equivalence class, which implies that $\M_{p,2,r}$ is a symmetric (and hence commutative) scheme in these cases. 

Now suppose $p \geq 2r+2$. let $M, M' \in \M_{p,2,r}$ be the associates corresponding respectively to the equivalence classes $X : (\overline{2}, 2'^{(r-1)})$ and $X' : (1^+,1^-, 2'^{(r-1)})$ (where the superscripts denote multiplicities). Now consider the matchings
\begin{eqnarray*}
S &:=& \set{ \set{2i-1, 2i} : i \in [r]}, \\
T &:=& \set{1,3} \cup \set{ \set{2i-1,2i} : i \in \set{3,4,\ldots, r+1} }.
\end{eqnarray*}
Then $(MM')[S,T] = 0$, and $(M'M)[S,T] = 2$. Since $MM' \neq M'M$, $\M_{p,2,r}$ is not commutative.
\end{proof}

While $\M_{p,q,r}$ is not symmetric in general, we have seen in Section~\ref{secLandP} that we can obtain symmetric subschemes of it (such as $\tilde{\M}_{p,q,r}$ and $\overline{\M}_{p,q,r}$) by contracting associates. Herein, we investigate the possibility of obtaining other symmetric subschemes of $\M_{p,q,r}$.

Given an h.c.c.\ that is not symmetric, a reasonable first attempt might be to take every matrix $B$ in the h.c.c.\ that is not symmetric, and contract  $\set{B, B^{\top}}$. While this preserves the properties (A1)-(A3), (A4) may no longer hold. For an example, the proof of Proposition~\ref{propNoncomm} is based on two symmetric associates whose product is not symmetric. Thus, any set of symmetric matrices containing these two associates would fail the spanning condition (A4). 

So, which are the contractions of $\M_{p,q,r}$ that do result in symmetric subschemes? In the special case of $q=r=2$, let $X_0, X_1, \ldots, X_9$ denote the equivalence classes corresponding to the partitions as in Figure~\ref{fig1}. Note that $X_0 : (2',2')$ corresponds to the identity matrix, and we have a scheme with 9 associates (when $p \geq 8$). For convenience, in this section, we will refer to the matrices in $\M_{p,2,2}$ as  $M_0, \ldots, M_9$ (instead of $M_{p,2,2,0},\ldots, M_{p,2,2,9}$), when the value of $p$ is clear from the context.

Recall the $B_i$ matrices defined before Proposition~\ref{HtildeHoverline} that correspond to contracting associates based on the number of vertices the union of the matchings saturate. Then we have
\begin{eqnarray*}
B_4 &=& M_0 + M_1,\\
B_5 &=& M_2+M_3,\\
B_6 &=& M_4+M_5+M_6+M_7,\\
B_7 &=& M_8 \\
B_8 &=& M_9.
\end{eqnarray*}
Also, we know from Proposition~\ref{HtildeHoverline} that
\[
\tilde{\M}_{p,2,2} = \overline{\M}_{p,2,2} = \set{I, M_1, B_5, B_6, B_7, B_8}
\]
is a 5-associate symmetric subscheme of $\M_{p,2,2}$ for all $p \geq 8$. Of course, there is also the 1-associate trivial subscheme. To investigate if there are any other contractions of $\M_{p,2,2}$ that also result in symmetric subschemes, we look into how the associates of $\M_{p,2,2}$ interact with each other. As shown in the proof of Proposition~\ref{propNoncomm}, not all pairs of these matrices in $\M_{p,2,2}$ commute when $p \geq 6$. In fact, some of these matrices commute for some values of $p$ but not others. 

Table~\ref{table1} shows the commutativity data for the h.c.c.\ $\set{M_i}_{i=0}^{9}$ for up to $p = 15$.  A checkmark ($\checkmark$) indicates that the matrices commute for all $p \leq 15$. A number indicates that those two matrices $M_i,M_j$ only commute for that specific value of $p$, among values of $p \leq 15$ for which both $M_i,M_j$ are non-zero. For example, $X_4 : (1^+,1^-,2')$ and $X_8 : (1^+,1^-, \overline{2})$ correspond to matching unions that saturate $6$ and $7$ vertices respectively, and thus $M_4, M_8$ are both non-zero only when $p \geq 7$. Now the entry ``$9$'' in the table means that $M_4, M_8$ commute when $p=9$, and do not commute for any $p \in \set{7,8,10,11,12,13,14,15}$. Finally, a blank entry indicates the matrices do not commute for any $p \leq 15$ for which they are both non-zero.

\begin{table}[htb]
\begin{center}
\scalebox{0.75}{
{\small
\begin{tabular}{|l|c|c|c|c|c|c|c|c|c|c|}
\hline
$M_iM_j$ commute? &$(2', 2') $&$ (4')$&$ (\overline{2},2')$&$ (\overline{4})$&$  (1^+, 1^-, 2')$&$ (\overline{2},\overline{2})$&$(3^+,1^-)$&$(1^+,3^-)$& $(1^+,1^-,\overline{2})$&$  (1^+,1^+,1^+,1^-)$\\
\hline
$(2', 2')$& $\checkmark$& $\checkmark$&$\checkmark$ &$\checkmark$ &$\checkmark$ &$\checkmark$ &$\checkmark$ &$\checkmark$ &$\checkmark$ &$\checkmark$\\
\hline
$ (4')$&$\checkmark$ & $\checkmark$& $\checkmark$& $\checkmark$& & && &$\checkmark$ &$\checkmark$ \\
\hline
$ ( \overline{2},2')$&$\checkmark$ & $\checkmark$& $\checkmark$& $\checkmark$& & && &$\checkmark$ &$\checkmark$ \\
\hline
$( \overline{4})$&$\checkmark$ & $\checkmark$& $\checkmark$& $\checkmark$& 6&6 &6 &6 &$\checkmark$ &$\checkmark$ \\
\hline
$(1^+, 1^-, 2')$& $\checkmark$& & & 6&$\checkmark$ &6 & & & 9& \\
\hline
$( \overline{2}, \overline{2})$& $\checkmark$& & &6 &6 &$\checkmark$ && &9 & \\
\hline
$(3^+,1^-)$&$\checkmark$ & & & 6& & & $\checkmark$&  & 9& \\
\hline
$(1^+,3^-)$&$\checkmark$ & & & 6& & &  & $\checkmark$ & 9& \\
\hline
$(1^+,1^-, \overline{2})$&$\checkmark$ & $\checkmark$& $\checkmark$& $\checkmark$& 9& 9& 9& 9&$\checkmark$ &$\checkmark$ \\
\hline
$(1^+,1^+,1^-,1^-)$&$\checkmark$ & $\checkmark$& $\checkmark$& $\checkmark$& & & & &$\checkmark$ &$\checkmark$ \\
\hline
\end{tabular}
}
}
\caption{Commutativity data for matrices in $\M_{p,2,2}$ for $p \leq 15$.}\label{table1}
\end{center}
\end{table}

We have also exhaustively tested all possible contractions of $\M_{p,2,2}$ for $p \leq 15$ to see which contractions result in symmetric, commutative subschemes, and found the following:

\begin{proposition}\label{Hn22groupings}
The following is an exhaustive list of all symmetric (and thus commutative) subschemes of $\M_{p,2,2}$, for $6 \leq p \leq 15$.
\begin{itemize}
\item[(i)]
The following are symmetric subschemes of $\M_{p,2,2}$ for all $p$ where $6 \leq p \leq 15$, 
\begin{itemize}
\item
the scheme made up of the non-zero matrices in the set \\ $\tilde{\M}_{p,2,2} = \set{I, M_1, B_5, B_6, B_7, B_8}$;
\item
the trivial scheme $\set{ I, J-I}$;
\item
the $2$-associate scheme $\set{I, M_1, J- M_1 - I}$.
\end{itemize}
\item[(ii)]
The following are sets of matrices that are only symmetric subschemes for certain values of $p$:
\begin{center}
\begin{tabular}{|l|l|}
\hline
$p$ & \tn{Symmetric Subschemes}\\
\hline
$6$ 
& $\set{I, M_4, J-M_4-I}$\\
& $\set{I, M_2+M_5, J - M_2-M_5-I}$\\
& $\set{I,M_1+M_2+M_6+M_7, M_3+M_5, M_4} $\\
& $\set{I, M_1+M_3+M_4, M_2+M_5, M_6+M_7} $\\
& $\set{I, M_1+M_4, M_2+M_5, M_3, M_6+M_7} $\\
\hline
$7$
& $ \set{I, M_1, B_5+B_7, B_6}$\\
\hline
$8$
& $\set{I, M_1+B_8, B_5+B_6+B_7} $\\
& $\set{I, M_1, B_5+B_6+B_7, B_8} $\\
& $\set{I, M_1+B_8, B_5+B_7, B_6} $\\
& $\set{I, M_1, B_5+B_7, B_6, B_8} $\\
\hline
$9$
& $ \set{I, M_1+M_2+M_6+M_7+M_9,  M_3+M_4+M_8}$\\
& $\set{I, M_1, B_5+B_8, B_6+B_7} $\\
\hline
$11$
& $\set{I, M_1, B_5+B_8, B_6+B_7} $\\
\hline
$12$
& $\set{I,M_1,B_5+B_7, B_6+B_8}$\\
\hline
\end{tabular}
\end{center}
\end{itemize}
\end{proposition}

The three subschemes that work for all values of $p$ we checked are no surprises: The trivial scheme and $\tilde{\M}_{p,2,2}$ are expected, and the third scheme $\set{I, M_1, J-M_1-I}$ is in fact the wreath product $\K_1 \wr \K_2$, where $\K_1, \K_2$ are trivial schemes on ground sets of sizes $\binom{p}{4}$ and $3$, respectively. This can be shown using the same argument as in the proof of Proposition~\ref{HtildeHoverline}, while noting that $\K_2$ is equivalent to the perfect matching scheme $\M_{4,2,2}$. Thus, we see that these three would be symmetric subschemes of $\M_{p,2,2}$ for all $p \geq 6$.

For Proposition~\ref{Hn22groupings}(ii), the $B_i$ notation is used when the given subscheme can be resulted from contracting associates in $\tilde{\M}_{p,2,2}$. Notice that for $p \neq 6,9$ (in which Table~\ref{table1} showed there are some ``coincidental'' commutativity between associates that are not present in other values of $p$), all subschemes we obtained are contractions of $\tilde{\M}_{p,2,2}$. It would be interesting to know if it is indeed true that, for all $p \neq 6,9$, all symmetric subschemes of $\M_{p,2,2}$ are in fact subschemes of $\tilde{\M}_{p,2,2}$ (or $\tilde{\M}_{p,2,2}$ itself).

We finish this section by proving a result that shows that, if a certain contraction of associates in $\M_{p,q,r}$ produces a symmetric subscheme for all values of $p$ up to a certain point, then it is assured that this contraction would yield a symmetric subscheme of $\M_{p,q,r}$ for all $p$. As we have seen in Section~\ref{sec41}, with fixed $q, r$, the number of associates in $\M_{p,q,r}$ increases as $p$ increases from $qr$ to $2qr$, and remains constant for all $p \geq 2qr$. For the compactness of stating our results, for the rest of the section, we let $\I_{q, r}$ denote the collection of isomorphism classes in $\M_{2qr,q, r}$,  and we will think of $\M_{p,q,r}$ as a set of $|\I_{q,r}|$ matrices, some of which are all zeros when $p < 2qr$.

We first need the following lemma.

\begin{lemma}\label{lemnextension}
Let $C_1, C_2 \subseteq \I_{q, r}$  be subsets of isomorphism classes in $\M_{p,q,r}$, and define matrices $M_{p,1} := \sum_{i \in C_1} M_{p,q, r,i}$ and $M_{p,2} := \sum_{i \in C_2} M_{p,q, r,i}$. Also let $S_1, S_1', S_2, S_2' \in [p]_q^r$. If 
\begin{equation}\label{lemnextensioneq1}
(M_{p,1}M_{p,2})[S_1, S_1'] = (M_{p,1}M_{p,2})[S_2,S_2']
\end{equation}
holds for all $p \leq 3qr$, then~\eqref{lemnextensioneq1} holds for all integers $p$.
\end{lemma}

\begin{proof}
Given $S,S' \in [p]_q^r$, and a set of vertices $U \subseteq [p]$, define 
$f^U(S,S')$ to be the set of matchings $T \in [p]_q^r$ where 
\begin{itemize}
\item
$(S, T) $ belongs to an isomorphism class in $C_1$;
\item
$(T,S')$ belongs to an isomorphism class in $C_2$;
\item
The vertices saturated by $S,S'$ and $T$ are all contained in $U$.
\end{itemize}

Notice that $(M_{p,1}M_{p,2})[S,S'] = |f^{[p]}(S,S')|$. Thus, the hypothesis that~\eqref{lemnextensioneq1} holds for all $p \leq 3qr$ can be restated as
\[
\left| f^U(S_1,S_1')\right| = \left| f^U(S_2,S_2') \right|, \quad \forall U \subseteq [p], |U| \leq 3qr.
\]
Now notice that, for arbitrary $p' \geq 3qr$,
\[
(M_{p',1}M_{p',2})[S_1,S_1'] = \left| f^{[p']}(S_1,S_1') \right| = \left| \bigcup_{U \subseteq [p'], |U|= 3qr} f^{U}(S_1,S_1') \right|.
\]
The last equality follows since the union of any 3 matchings $S_1,S_1',T \in [p]_q^r$ saturates at most $3qr$ vertices, so every matching in $f^{[p']}(S_1,S_1')$ is accounted for in the union. Next, one can apply the principle of inclusion-exclusion to express $\left| \bigcup_{U \subseteq [p'], |U|= 3qr} f^{U}(S_1,S_1') \right|$ as a linear combination of $\left| f^W(S_1,S_1') \right|$'s where $W$ is an intersection of sets of size $3qr$ (and thus has size no more than $3qr$). Thus, we obtain integers $b_W$'s such that 
\[
\left| \bigcup_{U \subseteq [p'], |U|= 3qr} f^{U}(S_1,S_1') \right| = \sum_{W \subseteq [p'], |W| \leq 3qr} b_W \left|  f^{W}(S_1,S_1') \right|.
\]
Notice that the coefficients $b_W$ only depend on $p', q$, and $r$, and not $S_1, S_1'$. Thus, by the same rationale we obtain that 
\[
(M_{p',1}M_{p',2})[S_2,S_2'] = \sum_{W \subseteq [p'], |W| \leq 3qr} b_W \left|  f^{W}(S_2,S_2') \right|.
\]
By our hypothesis, $\left| f^{W}(S_1,S_1')\right| = \left| f^{W}(S_2,S_2') \right|$ for all $W$ of size no more than $3qr$. Thus, we conclude that~\eqref{lemnextensioneq1} indeed holds for all $p' > 3qr$.
\end{proof}

Finally, let $X_0 \in \I_{q, r}$ denote the isomorphism class that corresponds to the identity matrix. Then we have the following:

\begin{proposition}\label{propnextension}
Let $C_1, \ldots, C_m$ be a partition of the non-identity isomorphism classes $\I_{q, r} \setminus \set{X_0}$. Define matrices
\[
B_{p,i} := \sum_{j \in C_i} M_{p,q, r,j}
\]
for every $i \in [m]$ and $p \geq qr$. If
\[
\B_p := \set{I} \cup \set{B_{p,i} : i \in [m]}
\]
is a symmetric subscheme of $\M_{p,q,r}$ for all $p \leq 3qr$, then it is in fact a symmetric subscheme of $\M_{p,q,r}$ for all $p$.
\end{proposition}

\begin{proof}
For convenience, let $B_{p,0} := I$ throughout this proof. It is clear that $\B_p$ satisfies (A1) and (A3) in Definition~\ref{defnscheme}. We next prove that it also satisfies (A4). By hypothesis, we have
\begin{equation}\label{propnextensioneq1}
B_{p,i}B_{p,j} \in \tn{Span}~\B_p
\end{equation}
for all $i,j \in \set{0,\ldots, m}$ and $p \leq 3qr$. Now suppose for a contradiction that there is an integer $p' > 3qr$ where~\eqref{propnextensioneq1} fails. Since $\M_{p',q, r}$ is an h.c.c.\ and thus satisfies (A4), we know that $B_{p',i}B_{p',j} \in \tn{Span}~\M_{p',q, r}$. Thus, there must exist an index $\l$ and $S_1, S_1', S_2, S_2' \in [p]_q^r$ where
\[
B_{p',\l}[S_1,S_1'] = B_{p',\l}[S_2,S_2'] \quad \tn{and} \quad (B_{p',i}B_{p',j})[S_1,S_1'] \neq (B_{p',i}B_{p',j})[S_2, S_2'].
\]
However, by our hypothesis, $B_{p,i}B_{p,j} \in \tn{Span}~\B_p$ for all $p \leq 3qr$. Thus,
$B_{p,\l}[S_1,S_1'] = B_{p,\l}[S_2,S_2']$, which implies that $(B_{p,i}B_{p,j})[S_1,S_1'] = (B_{p,i}B_{p,j})[S_2, S_2']$. Then by Lemma~\ref{lemnextension}, it must be the case that $(B_{p',i}B_{p',j})[S_1,S_1'] = (B_{p',i}B_{p',j})[S_2, S_2']$ as well. Thus, $B_{p', i}B_{p',j} \in \tn{Span}~\B_{p'}$. 

Next, we show that for all $p' > 3qr$, $B_{p',i}$ is a symmetric matrix. By assumption,
\begin{equation}\label{propnextensioneq2}
(B_{p,0}B_{p,i})[S,S'] = (B_{p,0}B_{p,i})[S',S]
\end{equation}
 for all $p \leq 3qr$. Thus, applying Lemma~\ref{lemnextension} again, we obtain $B_{p',i} = B_{p',i}^{\top}$ for all $p' > 3qr$ as well. It then follows that (A2) holds as well. This finishes the proof.
\end{proof}

In the case of $q=r=2$, Proposition~\ref{propnextension} simply tells us that the three subschemes listed in Proposition~\ref{Hn22groupings}(i) are indeed subschemes of $\M_{p,2,2}$ for all $p$, which we have already discussed. It would be interesting to see if some version of the converse of Proposition~\ref{propnextension} is true --- that if a certain contraction fails to yield a symmetric subscheme for enough small values of $p$, then we can guarantee that it would also fail to do so for large $p$.

\section{Concluding remarks}

Throughout this paper, we have pointed out some connections between association schemes and the analyses of semidefinite programs, as illustrated mainly by studying the lift-and-project relaxations of several classical problems in combinatorial optimization. In particular, we saw that the process of verifying the positive semidefiniteness of a certificate matrix could be simplified if said matrix is related to an association scheme whose eigenvalues are known.

We comment that, since the hypermatching packing problem considered in Section~\ref{sec31} only concerns vertex saturation, two matchings are essentially interchangeable in the problem if they saturate the exact same set of vertices. Thus, instead of considering a matching of $r$ hyperedges in $K_p^q$, one could have worked with a single hyperedge of size $qr$. Then we would be working with the simpler scheme $\M_{p,qr, 1} = \J_{p, qr}$ instead. One of the reasons why we based our discussion on the more general framework of $\M_{p,q,r}$ is that this allows easier adaptation to study other combinatorial optimization problems where such a reduction may not be possible or suitable. 

Finally, the approach of using association schemes and h.c.c.s to help analyze lift-and-project relaxations could also benefit from a better understanding of the underlying schemes, as we attempted to do for $\M_{p,q,r}$ in Section~\ref{secHnkl}. For instance, for the perfect matching scheme $\M_{2r, 2, r}$, it is known that its eigenvalues can be determined using zonal polynomials~\cite{MacDonald95}. More recently, Srinivasan~\cite{Srinivasan20} showed that the eigenvalues of $\M_{2r,2,r}$ can be computed by recursively solving systems of linear equations that involve the central characters of the symmetric groups $\set{ \S_{2i} : i \in [r]}$. Also lately, there has been interest in studying the eigenvalues of the perfect matching derangement graph~\cite{GodsilM16, Lindzey17, KuW18}, whose adjacency matrix is the sum of a subset of associates in $\M_{2r,2,r}$. However, we still do not have explicit and tractable combinatorial descriptions of the eigenvalues of the scheme $\M_{2r,2,r}$ in general, and a breakthrough on this front could give us a better handle on the eigenvalues of the matrices in $\tn{Span}~\overline{\M}_{p,2,r}$, a broader class of potential certificate matrices than those in $\tn{Span}~\tilde{\M}_{p,2,r}$.

\bigskip

{\bf Acknowledgment:}
We thank an anonymous referee for very useful comments and suggestions which helped
improve the content and the presentation of the paper.

\bibliographystyle{plain}
\bibliography{ref} 
\end{document}